\documentclass{amsart}

\usepackage[utf8]{inputenc}

\usepackage[greek,english]{babel}

\usepackage{amsmath,amssymb,amsthm,stackengine}

\usepackage{esint}

\usepackage{babelbib}

\usepackage{hyperref}

\usepackage[T1]{fontenc}

\usepackage{lmodern}

\usepackage{xcolor}

\theoremstyle{definition}

\newtheorem{lemma}{Lemma}
\newtheorem{thm}{Theorem}
\newtheorem{rem}{Remark}
\newtheorem{prop}{Proposition}

\author{Glal Bacho}
\thanks{Glal Bacho: Lehrstuhl für Angewandte Analysis, RWTH Aachen University, Kreuzherrenstr. 2, 52062 Aachen, Germany. Email: \texttt{bacho@math1.rwth-aachen.de}. Phone: +49 241 80-94584.}

\author{Christof Melcher}
\thanks{Christof Melcher: Lehrstuhl für Angewandte Analysis, RWTH Aachen University, Kreuzherrenstr. 2, 52062 Aachen, Germany. Email: \texttt{melcher@rwth-aachen.de}. Phone: +49 241 80-94585.}

\title[Bubbling analysis of bimeron configurations]{Bubbling analysis of bimeron configurations}

\renewcommand{\d}[1]{\mathrm{d}#1}

\newcommand{\R}{\mathbb{R}}
\newcommand{\C}{\mathbb{C}}
\newcommand{\Z}{\mathbb{Z}}
\newcommand{\N}{\mathbb{N}}
\renewcommand{\S}{\mathbb{S}}
\renewcommand{\Re}{\operatorname{Re}}
\renewcommand{\Im}{\operatorname{Im}}

\newcommand{\tr}{\operatorname{tr}}

\newcommand{\T}{\mathbb{T}}
\newcommand{\D}{\mathbb{D}}
\newcommand{\e}[1]{\bs{\hat{e}}_{#1}}
\newcommand{\m}{\boldsymbol{m}}
\renewcommand{\c}{\boldsymbol{c}}
\newcommand{\h}{\boldsymbol{h}}
\newcommand{\tm}{\boldsymbol{u}}
\newcommand{\tms}{u}
\newcommand{\bm}{\overline{\m}}
\newcommand{\bms}{\overline{m}}

\newcommand{\bs}[1]{\boldsymbol{#1}}

\newcommand{\mDMI}{\left( \begin{array}{c}
         - \nabla m_3  \\
         \nabla \cdot m
    \end{array}\right)}

\newcommand{\St}{\mathbb{S}^2}
\newcommand{\dist}{\mathrm{dist}}

\newcommand{\loc}{\mathrm{loc}}

\newcommand{\dd}[2]{\frac{\partial #1}{\partial #2}}

\newcommand{\eps}{\varepsilon}
\newcommand{\del}{\partial}

\begin{document}

\date{\today}

%\dedicatory{Dedicated to Professor Willi Jäger on the occasion of his 85th birthday, \\ who first introduced one of us, three decades ago, to the mathematics of ferromagnetism and harmonic maps, and whose lasting inspiration continues to resonate across generations.}

\begin{abstract}
Chiral symmetry breaking in magnetic thin films stabilizes new families of topological solitons that are absent in conventional anisotropic ferromagnets. Beyond the classical chiral skyrmion, which arises in uniaxial systems through the Dzyaloshinskii–Moriya interaction, dipolar configurations known as chiral bimerons emerge in easy-plane situations. On compact domains, these solitonic states appear as localized concentrations of energy embedded in an easy-plane background. The analysis, inspired by the Sacks–Uhlenbeck theory for approximate harmonic maps, combines blow-up methods with quantitative rigidity for the Möbius group to describe the structure, scaling behaviour, and asymptotic limits of such magnetic solitons in both the conformal and large-domain regimes.
\end{abstract}

\maketitle

\section{Introduction}
Topological solitons in two-dimensional field theories arise as finite-energy critical points of variational functionals whose configuration spaces possess nontrivial homotopy structure. A standard example is provided by the $\mathrm{O}(3)$ sigma model, where one studies maps from $\mathbb{R}^{2}$ to $\S^{2}$, the unit sphere in $\mathbb{R}^{3}$, minimizing the Dirichlet energy subject to a topological constraint. For smooth configurations approaching a constant value at infinity sufficiently fast, this constraint can be formulated in terms of an integer-valued integral of the Jacobian, called the topological charge or degree (cf. \eqref{eq:degree}). Mathematically, these finite-energy critical points correspond to harmonic maps from $\S^{2} \cong \mathbb{C} \cup \{\infty\}$ to $\S^{2}$ of a given degree. The energy-minimizing harmonic maps in each homotopy class are the holomorphic (or antiholomorphic) ones, and they form a smooth finite-dimensional moduli space of rational functions of the corresponding algebraic degree, which parameterizes the soliton positions, sizes, and orientations.

Solitons of unit topological charge are commonly referred to as \emph{skyrmions}, named after T.\,H.\,R. Skyrme, who introduced a sigma-type model from $\mathbb{R}^3$ to $\mathrm{SU}(2) \cong \S^3$ with stabilizing higher-order interaction terms (the nuclear Skyrme model) as a nonlinear field theory admitting topological soliton solutions, see \cite{skyrme1961non}. Confirming the existence of nuclear skyrmions occurring as energy minimizers in non-trivial homotopy classes by means of the direct method of the calculus of variations is a prominent instance for the concentration-compactness principle, see, e.g., \cite{esteban2004existence, lin2004existence, lin2007energy}.

Fractional topological charges can also appear. The term \emph{meron} (from the Greek \emph{meros} \foreignlanguage{greek}{μέρος}, ``part'') denotes a field configuration carrying half the topological charge of a skyrmion. Merons were first introduced in Yang–Mills theory as half-instanton solutions with fractional topological charge \cite{Allan1977quarkconfinement}, see \cite{Glimm1978} for a more rigorous treatment. These ideas were subsequently adapted to two-dimensional $\mathrm{O}(3)$ sigma models and magnetic systems, see \cite{gross1978}. On the other hand, such configurations naturally arise, for example, in the setting of harmonic maps from the unit disk $\D \subset \mathbb{R}^2$ to the sphere $\S^2$, when one prescribes boundary data that trace the equator of $\S^2$ exactly once and that select either the upper or lower hemisphere as target. The resulting map covers only half of $\S^2$, and in this precise sense it represents a \emph{half-skyrmion} or \emph{meron} configuration. Uniqueness in each polarity class (upper or lower hemisphere) is an instance of the seminal work by Jäger and Kaul \cite{jager1979uniqueness, kaul1983rotationally}. 

More generally, for such maps $\m = (m_1, m_2, m_3): \D \to \S^2$, a well-studied class of variational problems motivated by thin-film micromagnetics is concerned with Ginzburg--Landau type energies
\begin{equation} \label{eq:GL}
E_{0,\eps}(\m) = \frac{1}{2} \int_{\D} |\nabla \m|^2 + \left( \frac{m_3}{\eps} \right)^2 \, \mathrm{d}x
\end{equation}
and minimizers subject to certain Dirichlet boundary conditions in the form of equator maps. Depending on the winding number of the boundary maps, the system develops magnetic vortices (which we may alternatively call merons) that asymptotically ($\eps \ll 1$) interact according to a renormalized Coulomb-type energy, as in the $\mathbb{C}$-valued case introduced in \cite{BBH}. The static case has been examined in \cite{LinHang2001PlanarFerro}, and the dynamics in \cite{KMMS, kurzke2014vortex}.
A somewhat degenerate situation occurs for constant boundary data: 
no non-trivial equilibria exist beyond the constant absolute minimizer. 
For finite $\varepsilon>0$, this follows from the Pohozaev identity
\begin{equation}\label{eq:Pohozaev}
\int_{\D} \left( \frac{m_3}{\varepsilon} \right)^2 \, \mathrm{d}x 
= \frac{1}{2} \int_{\S^1} 
\left( \left| \dd{\m}{\tau} \right|^2 
- \left| \dd{\m}{\nu} \right|^2 
+ \left( \frac{m_3}{\varepsilon} \right)^2 \right) \, \mathrm{d}s .
\end{equation}
Under constant boundary conditions in $\S^1 \times \{0\}$, the boundary integral on the
right-hand side is non-positive, hence $m_3 \equiv 0$ in $\D$.
Thus $\m$ is an $\S^1$-valued harmonic map. 
Lifting to a harmonic phase and applying the maximum principle yields
that the phase is constant, and therefore $\m$ is constant.
In the limiting case $\varepsilon = \infty$, the problem reduces to
harmonic maps into $\S^2$. A conformality argument shows that any
finite-energy solution is constant
(see, e.g., \cite{lemaire1978applications,Brezis1983harmonic,helein2002harmonic}).

\subsubsection*{Dzyaloshinskii--Moriya interaction}

A recent approach to stabilizing nontrivial magnetic equilibria introduces symmetry-breaking interaction terms that are linear combinations of the so-called \emph{Lifshitz invariants},
\[
m_j \nabla m_k - m_k \nabla m_j,
\]
which arise in models of nematic liquid crystals and various spin--orbit-coupled condensed-matter systems. In chiral magnets, these invariants represent the antisymmetric exchange known as the \emph{Dzyaloshinskii--Moriya interaction} (DMI). Having the structure of a differential one-form coupling, the DMI is inherently sensitive to orientation: it breaks independent rotational symmetries in both domain and target, favors modulated spin textures, and lowers the total energy by selecting a preferred sense of rotation, or chirality, of the magnetization field.
Writing $\m=(m,m_3)$, a specific form of DMI relevant for magnetic thin films is given by
\begin{equation} \label{eq:DMI}
e_{\mathrm{DMI}}(\m) = \frac{1}{2}\big( m_3 (\nabla \cdot m) - (m \cdot \nabla) m_3 \big).
\end{equation}
Since
\[
e_{\mathrm{DMI}}(\m) = (\nabla \cdot m)\, m_3 - \tfrac{1}{2}\nabla \cdot (m\, m_3),
\]
the divergence term in \eqref{eq:DMI} can typically be discarded under suitable boundary conditions such as $m=0$, $m_3=0$, or periodicity. The remaining term energetically favors, in an axisymmetric setting, hedgehog or anti-hedgehog configurations of $m$ depending on the sign of $m_3$.
Another important variant, occurring in cubic helimagnets and cholesteric liquid crystals, 
is the \emph{helicity} form
\begin{equation} \label{eq:helicity}
e_{\mathrm{DMI}}(\m) = \m \cdot (\nabla \times \m),
\end{equation}
which promotes whirling spin structures and spiralization. It agrees with the well-known helicity functional, which plays a significant role in various branches of mathematical physics, particularly in topological fluid dynamics. Perturbing the harmonic map problem by helicity yields an optimality condition of Beltrami type, $\nabla \times \m + \m = 0$, whose unimodular solutions correspond to one-dimensional helical states (see~\cite{melcher2014chiral} and references therein). 
Finally, DMI-type interactions can also emerge geometrically as curvature-induced effects, revealed naturally through Cartan’s moving frame calculus~\cite{melcher2019curvature}. In particular, curvature-driven DMI on spherical surfaces with normal anisotropy gives rise to novel saddle-point configurations~\cite{gustafson2025saddle}, which can be interpreted as precursors of chiral skyrmions discussed in the following paragraph.

\subsubsection*{Chiral skyrmions}
When sufficiently strong uniaxial anisotropy or Zeeman coupling is present, thereby defining a vacuum state at infinity, the DMI contribution is balanced by the Dirichlet energy, ensuring positivity of the total energy. In this ferromagnetic regime, isolated \emph{chiral skyrmions} arise as local energy minimizers of topological degree $\pm1$, depending on the orientation near infinity. The key element of the stabilization mechanism, and the foundation of the existence theory, is that DMI lowers the skyrmion energy strictly below the critical threshold $4\pi$, the energy of a unit-degree harmonic bubble. This enables a variational existence proof in the spirit of Brezis and Coron~\cite{Brezis1983harmonic}, first established in~\cite{melcher2014chiral} (see also~\cite{hoffmann2017antiskyrmions} for the possible stabilization of anti-skyrmions).  
Rotational invariance of the functional about the anisotropy axis suggests that minimizing skyrmions should retain axisymmetry, a question that remains among the most prominent open problems in the field. Spectral stability and local minimality of axisymmetric skyrmions within the symmetry-breaking energy space were established in~\cite{LiMelcherStability, gustafson2021co}.  
The \emph{conformal limit} in the regime of small DMI, introduced in~\cite{doring2017compactness}, provides a systematic framework for extracting universal information about the structure and energetics of chiral skyrmions in the regime of vanishing core size, where the conformal Dirichlet interaction dominates. In this limit, a symmetric harmonic bubble of degree $\pm1$ arises that represents the skyrmion core. A central advance in this context is the quantitative rigidity result for unit-degree harmonic maps (cf. \eqref{eq:rigidity} below) obtained by Bernand-Mantel, Muratov, and Simon in~\cite{MuratovSimonRigidity}, which yields the requisite $\Gamma$ equivalence; that is, matching upper bounds and ansatz-free lower bounds for the energy asymptotics. This characterization identifies both the universal limiting profile and the asymptotic scaling of the energy and skyrmion core (or radius), as established for axisymmetric chiral skyrmions in~\cite{komineas2020profile, gustafson2021co}.  Finally, the Dirichlet problem for chiral skyrmions under confinement is examined in~\cite{monteil2023magnetic}, where skyrmions are shown to concentrate as point-like states, and their positions and sizes are determined by a boundary-repelling renormalized energy.

\subsubsection*{Chiral Bimerons}
Consider a disk-shaped chiral easy-plane ferromagnet governed by the energy functional
(writing $\m=(m,m_3)$)
\begin{equation}\label{eq:full_energy}
E_{\lambda,\varepsilon}(\m) =
\int_{\mathbb{D}} 
\frac{|\nabla \m|^2}{2} 
+ \lambda \, (\nabla \cdot m)\, \frac{m_3}{\varepsilon}  
+ \frac{1}{2} \left( \frac{m_3}{\varepsilon} \right)^2 \, \mathrm{d}x ,
\end{equation}
which extends~\eqref{eq:GL} by inclusion of a properly scaled DMI term 
$\tfrac{\lambda}{\varepsilon} (\nabla \cdot m) m_3$ of the form \eqref{eq:DMI}.
The imposition of a constant in-plane Dirichlet boundary condition 
\begin{equation} \label{eq:BC}
\m = (c,0) \in \S^1 \times \{0\} \quad \text{on} \quad \mathbb{S}^1
\end{equation}
necessarily breaks the rotational symmetry of the variational problem.
This symmetry breaking gives rise to configurations consisting of a bound pair of merons 
with opposite in-plane winding and opposite polarity, commonly referred to as (chiral) \emph{bimerons} 
see, e.g.,~\cite{bachmann2023meron, Gao2019topologicalmeron, gobel2021beyond, Go2019MagneticBimerons, Lin2015fractionalskyrmion}.

In the vortex language, a bimeron represents a dipolar structure, that is, a vortex--antivortex pair 
which does not occur in the unperturbed model~\eqref{eq:GL} without DMI.
Topologically, a bimeron carries the same skyrmion number as a conventional skyrmion; 
the distinction lies in the geometric arrangement of the magnetization rather than in its topological charge.
In contrast, however, the invariance of~\eqref{eq:full_energy} under the transformation 
$\m(x) \mapsto (m,-m_3)(-x)$
implies that, for the same boundary condition \eqref{eq:BC}, bimeron configurations of the same energy but opposite degree replicate.
When considered on a flat torus~$\T^2$, such configurations correspond to meron lattices 
as encountered in~\cite{li2020lattice} through symmetry-breaking bifurcation methods 
applied to a chiral Ginzburg--Landau model, thereby connecting with physical observations 
reported in~\cite{yu2018transformation}.

The existence of degree $\pm 1$ local minimizers for~\eqref{eq:full_energy} can be studied in a framework analogous to the skyrmion case analyzed in~\cite{melcher2014chiral}, 
relying on a concentration--compactness argument and the subcritical energy bound $\inf E_{\lambda,\varepsilon} < 4\pi$ in topological sectors of degree $\pm 1$. 
For chiral bimerons on $\mathbb{R}^2$, a new difficulty arises from the interplay between DMI and easy-plane anisotropy when attempting to rule out vanishing. 
This challenge was elegantly addressed in the recent work of Deng, Ignat, and Lamy~\cite{DengRaduLamyConformalBimeron} (focusing on helicity-type DMI \eqref{eq:helicity}), 
who adapted the approach of~\cite{MuratovSimonRigidity}. 
Their method leverages quantitative rigidity results for the Möbius group, showing that any finite-energy map $\m:\mathbb{R}^2 \to \mathbb{S}^2$ of unit degree is, in a precise functional sense, close to a Möbius map $\boldsymbol{\Psi}$ after stereographic projection. 
The rigidity result of~\cite{MuratovSimonRigidity} (see also \cite{hirsch2022note, ToppingRigidity}) takes the form
\begin{equation} \label{eq:rigidity}
\int_{\mathbb{R}^2} \bigl| \nabla (\m - \boldsymbol{\Psi}) \bigr|^2 \, \mathrm{d}x 
\;\lesssim\; \frac{1}{2} \int_{\mathbb{R}^2} |\nabla \m|^2 \, \mathrm{d}x - 4\pi.
\end{equation}
By adapting this result to the easy-plane setting and employing optimal Möbius maps, they derived sharp lower bounds, controlled the scale of minimizers in the small-DMI regime, and rigorously described the convergence of energy-minimizing bimerons to conformal maps as the DMI parameter tends to zero.\\

Here we examine chiral bimeron configurations on compact domains, 
specifically on the unit disk $\D$ with Dirichlet boundary conditions~\eqref{eq:BC}, 
and on the flat torus $\T^2$. 
In Section~\ref{sec:energybounds}, we present the corresponding lower and upper energy bounds, 
based on constructing suitable trial fields on the disk to ensure the compactness of minimizing sequences. 
After stereographic projection, the construction employs the bimeron prototype
\begin{equation} \label{eq:ansatz}
f(z) = c\,\frac{z - a}{z + a}, \qquad a \in \R^+,
\end{equation}
which carries a vortex at $z = a$ and an antivortex at $z = -a$, and is smoothly cut off in an annular region. 
These constructions provide configurations with sufficiently small energy to guarantee the compactness of minimizing sequences. 
Consequently, we obtain the existence result (Theorem~\ref{thm:existence})
also valid in the periodic case.

\medskip

In Section~\ref{sec:conformallimit} we present an approach to the conformal limit, 
embedding it into the framework of bubbling for approximate harmonic maps 
as developed in~\cite{DingTian1995Approximateharmonicmap, qing1997bubbling, lin1998energy} 
and extended to include Dirichlet boundaries in~\cite{Jost2019NoNeckDirichlet}. 
The blow-up procedure fits naturally within this framework for maps with $L^2$-bounded tension fields, 
where no-neck results—both for energy and oscillation—remain valid even under Dirichlet boundary conditions. 
The decay of neck energy ensures compactness and controls concentration, providing a natural setting 
for the analysis of parameter-dependent regimes. 
Theorem~\ref{thm:bubbling} shows that bimerons in easy-plane chiral films are stabilized 
by the Dzyaloshinskii--Moriya interaction: they exist for every fixed $0 < |\lambda| < 1$, and, 
in the weak-chirality limit $\lambda \to 0$, their cores converge (after rescaling) 
to a harmonic profile which, after translation and rotation, is stereographically represented by 
\eqref{eq:ansatz}.

\medskip

In Section~\ref{sec:Energygap} we introduce and analyze an auxiliary problem arising as an effective easy-plane limit of~\eqref{eq:full_energy} as $\eps \to 0$. 
Suppressing the out-of-plane component leads to the anisotropic harmonic map functional
\[
E_\lambda(m) = \frac{1}{2} \int_{\T^2} |\nabla m|^2 - \lambda^2  (\nabla \cdot m)^2 \,\mathrm{d}x,
\]
for maps $m:\T^2 \to \S^1$. 
Comparable energies, obtained by relaxing the $\S^1$ constraint via a Chern-Simons-Higgs potential, 
arise in models of nematic--isotropic transitions with highly disparate elastic constants; see~\cite{golovaty2020model}. 
Critical points satisfy a quasilinear perturbation of the classical harmonic map equation 
$\Delta m + |\nabla m|^2 m = 0$. 
For $0 < |\lambda| < 1$, the system remains uniformly elliptic, although the ellipticity constants and the structure of the principal part depend on the direction of $m$, 
and the additional term involves the tangential projection of $\nabla(\nabla \cdot m)$ onto $\T_m \S^1$. 
In the critical case $\lambda^2 = 1$, however, the energy supports singular vortex structures. 
Combining a Pohozaev-type identity with a Sacks--Uhlenbeck perturbation scheme, adapted to the anisotropic setting, 
yields the energy gap stated in Theorem~\ref{thm:energygap}, which plays a central role in the large-domain analysis.

\medskip

In Section~\ref{sec:eps0} we finally consider the large-domain limit $\eps \to 0$ of unit-degree minimizers of \eqref{eq:full_energy} on the torus, obtained in Theorem~\ref{thm:existence} for fixed DMI coupling $0 < |\lambda| < \lambda_0 \ll 1$, thereby producing bimeron configurations on $\R^2$ (Theorem~\ref{thm:R2existence}). 
The energy gap established in Theorem~\ref{thm:energygap} constrains the behavior of minimizers away from the skyrmion core, ensuring that the far field relaxes to a uniformly easy-plane state and that all nontrivial topology remains confined to a localized core region. 
Uniform $\eps$-independent regularity bounds are obtained via higher-order energy estimates combined with a variant of Schoen’s trick. 
The intrinsic core scale $R_{\lambda,\eps}$ of a minimizer $\m_{\lambda,\eps}$ satisfies the asymptotic relation $R_{\lambda,\eps} \sim \eps$. 
Crucially, as a consequence of quantitative rigidity, the neck energy between the core and the far field decays, providing a robust alternative to the concentration-compactness principle.

\section{Energy bounds and existence}\label{sec:energybounds}
Crucial ingredients for compactness of minimizing sequences and existence results 
are energy bounds from below and above.  
A starting point is the classical topological lower bound
\begin{equation}\label{eq:topo-lower}
    \frac{1}{2} \int_{\D} |\nabla \m|^2\, \mathrm{d}x 
    \;\geq\; 4\pi\, |\deg \m|
\end{equation}
for finite-energy maps $\m : \D \to \S^2$ with boundary values \eqref{eq:BC},
see e.g.~\cite{Brezis1983harmonic}.  
For such maps, the Brouwer degree is well-defined and given by
\begin{equation}\label{eq:degree}
    \deg \m 
    = \frac{1}{4\pi}\int_{\D} 
        \m \cdot (\partial_1 \m \times \partial_2 \m)\, \mathrm{d}x.
\end{equation}
Hence the configuration space 
\[
H^1_c(\D; \S^2) =\{ \m \in  H^1(\D;\S^2): m-c \in H^1_0(\D; \R^2)\} 
\]
decomposes into closed topological sectors 
labeled by $\deg \m \in \mathbb{Z}$.
Taking into account the estimate
\begin{equation}\label{eq:lambda-ineq}
    \Bigl| \lambda (\nabla \!\cdot\! \m)\, \frac{m_3}{\varepsilon} \Bigr| 
    \;\le\;
    \frac{\lambda^2}{2}|\nabla \m|^2 
    + \frac{m_3^2}{2\varepsilon^2},
\end{equation}
one easily deduces the following lower bounds for the full energy 
$E=E_{\lambda, \eps}$ (cf.~\eqref{eq:full_energy}) with $0<|\lambda|<1$ and 
$\m \in H^1_c(\D;\S^2)$:
\begin{equation}\label{eq:toplowerbound}
    E(\m) 
    \;\geq\; \frac{1-|\lambda|}{2} 
    \int_{\D} \left( |\nabla \m|^2 + \frac{m_3^2}{\varepsilon^2} \right) 
    \mathrm{d}x
\end{equation}
and 
\[
E(\m) 
   \;\geq\; \frac{1-\lambda^2}{2} \int_{\D} |\nabla \m|^2 \, \mathrm{d}x.
\]
In particular, the topological lower bound \eqref{eq:topo-lower} implies
\begin{equation}\label{eq:topo-lower-lambda}
    E(\m) 
    \;\geq\; 4\pi (1-\lambda^2)\, |\deg \m|.
\end{equation}

Exact topological lower bounds of the form 
$E(\m) \ge 4\pi\, |\deg \m|$
can be obtained in models for chiral skyrmions, 
where the easy-plane anisotropy is replaced by an easy-axis term 
or combined with a Zeeman interaction. 
Such bounds arise naturally in gauge-field inspired formulations 
that lead to classes of exactly solvable models, 
see e.g.~\cite{schroers1995bogomol,melcher2014chiral, doring2017compactness, barton2020magnetic}. \\

We are interested in the variational problem for $E=E_{\lambda, \eps}$ (cf.~\eqref{eq:full_energy})
\begin{equation}\label{eq:variation}
    I 
    = \inf\Bigl\{ 
        E(\m) : \m \in H^1(\D;\S^2),\
        \m|_{\mathbb{S}^1} = (c,0) \in \S^1 \times \{0\},\
        \deg \m = -1
    \Bigr\}
\end{equation}
obtaining the following upper bounds:

\begin{prop} \label{prop:upper_bound}
Let $\eps_0>0$. For all $\lambda>0$ there exists $c(\lambda)>0$ such that  
\[
I \le 4 \pi \left(1- c(\lambda) \right)
\]
for all $\eps \in (0,\eps_0)$ and
\[
c(\lambda) \ge \frac{\lambda^2}{8|\ln \lambda|}  \left( 1 + o(1) \right)
\quad \text{as} \quad \lambda \to 0.
\]
\end{prop}

The crucial feature of the chiral interaction is that it can push the minimal energies in the non-trivial homotopy class below the critical threshold of $4\pi$. This is sufficient to deduce the existence of local energy minimizers via the concentration-compactness principle.

\begin{thm} \label{thm:existence}
For all $\eps>0$ and $0<|\lambda|<1$ the infimum of $E_{\lambda, \eps}(\m)$ in the class of finite energy fields $\m: \D \to \S^2$ with Dirichlet boundary condition 
\begin{align*}
\m|_{\S^1}=(c,0)\in \S^1\times\{0\}
\end{align*}
and $\deg \m = \pm 1$ is attained by smooth fields $\m$ with $E_{\lambda, \eps}(\m)< 4 \pi$.
\end{thm}

Adapting the strategy from \cite{MuratovSimonRigidity}, sharp lower bounds, 
which significantly improve upon \eqref{eq:topo-lower-lambda},
were obtained in \cite{DengRaduLamyConformalBimeron}. In combination with the corresponding 
upper bounds, these results provide the correct energy asymptotics 
in the conformal limit $\lambda \to 0$ in the full-space situation, 
first obtained in the context of chiral skyrmions in 
\cite{komineas2020profile, komineas2023chiral, MuratovSimonRigidity, gustafson2021co}.

\subsection{Upper bound construction}
We shall focus on the case $\deg \m =-1$.
Lifting $\m = \Phi(u)$ via the stereographic map
\begin{equation} \label{eq:stereographic}
    \Phi : \R^2 \to \S^2, 
    \qquad 
    \Phi(y) = \left( \frac{2y}{|y|^2 + 1},\, \frac{|y|^2 - 1}{|y|^2 + 1} \right),
\end{equation}
to a field $u : U \to \R^2$ on an open set $U \subset \R^2$ where $\m$ avoids the north pole 
$\e{3} = (0,0,1)$, a straightforward computation, carried out in Appendix~\ref{sec:proofdensitylift}, yields the pulled-back energy density
\begin{align*}
    \tilde{e}(u)
    &= 
    2\,\frac{|\nabla u|^2}{(|u|^2 + 1)^2}
    - 2\,\frac{\lambda}{\varepsilon}\,
        \frac{\nabla \!\cdot\! u}{(|u|^2 + 1)^2}
    + \frac{1}{2\varepsilon^2}\,
        \frac{(|u|^2 - 1)^2}{(|u|^2 + 1)^2}
    + \frac{\lambda}{\varepsilon}\, 
        \nabla \!\cdot\! \frac{u}{|u|^2 + 1}.
\end{align*}

Discarding the divergence term gives the \emph{reduced energy density}
\begin{equation}\label{eq:reduced-density}
    e(u)
    :=
    2\,\frac{|\nabla u|^2}{(|u|^2 + 1)^2}
    - 2\,\frac{\lambda}{\varepsilon}\,
        \frac{\nabla \!\cdot\! u}{(|u|^2 + 1)^2}
    + \frac{1}{2\varepsilon^2}\,
        \frac{(|u|^2 - 1)^2}{(|u|^2 + 1)^2}.
\end{equation}

Note that formula~\eqref{eq:reduced-density} is to be applied on 
$\D \setminus \m^{-1}(\{\e{3}\})$, since $|u| \to \infty$ as 
$\m \to \e{3}$.  
In our setting, there is only one point in $\m^{-1}(\{\e{3}\})$. 
The energy density behaves particularly well when the complex field 
$u_1 + i u_2$ is holomorphic, which greatly simplifies subsequent computations.
Indeed, setting $f = u_1 + i u_2$ and using the Cauchy--Riemann equations, 
we have
\begin{align*}
    f' 
    &= \partial_1 u_1 + i \partial_1 u_2 
     = \partial_2 u_2 - i \partial_2 u_1.
\end{align*}
Hence with $\nabla^{\perp} = (-\partial_2, \partial_1)$
\begin{equation*}
    2|f'|^2 = |\nabla u|^2,
    \quad \text{and} \quad  
    f' = \tfrac{1}{2}\bigl(\nabla \!\cdot\! u 
    + i\, \nabla^{\perp}\! \cdot u \bigr).
\end{equation*}

\begin{lemma}\label{lemma:densityholomorph}
For $f = u_1 + i u_2$ holomorphic, 
\eqref{eq:reduced-density} takes the form
\begin{equation*}
    e(u)
    = 4\,\frac{|f'(z)|^2}{(|f(z)|^2 + 1)^2}
      - 4\,\frac{\lambda}{\varepsilon}\,
        \Re\!\left( 
            \frac{f'(z)}{(|f(z)|^2 + 1)^2}
        \right)
      + \frac{1}{2\varepsilon^2}\,
        \frac{(|f(z)|^2 - 1)^2}{(|f(z)|^2 + 1)^2}.
\end{equation*}
\end{lemma}
We will henceforth write $e(f)$ for the right-hand side of the above identity.

\subsubsection*{Bimeron ansatz}

Consider the meromorphic ansatz \eqref{eq:ansatz} with vortices centered on the real axis and midpoint at the origin,
which corresponds to a vortex located at $(a,0)$ and an antivortex at $(-a,0)$ for some $a>0$; see, for instance, \cite{bachmann2023meron, gross1978}. This ansatz does not satisfy the prescribed boundary condition and therefore requires a cut-off modification, as described in Section~\ref{subsec:cut-off}. On the disk $D_R \subset \D$ with radius $R>0$, the energy is evaluated as a function of the parameters $a>0$ and $R>0$. For convenience, the constant $c$ is set equal to $1$.

\begin{lemma}\label{lemma:localenergy}
Let $f(z) = \frac{z - a}{z + a}$.  
Then the energy on $D_R$ is given by
\begin{equation*}
\int_{D_R} e(f)\, \mathrm{d}z
= 4\pi \!\left(1 - \frac{\lambda a}{2\eps}\right)
  \frac{R^2}{R^2 + a^2}
  + \pi \frac{a^2}{\eps^2}
    \!\left(
      \ln\!\frac{R^2 + a^2}{a^2}
      - \frac{R^2}{R^2 + a^2}
    \right)\!.
\end{equation*}
\end{lemma}

\begin{proof}
In order to apply Lemma~\ref{lemma:densityholomorph}, it is necessary to verify that the energy density extends smoothly to $\C$.  
Since $f'(z) = \frac{2a}{(z + a)^2}$, it follows for $z \neq -a$ that
\begin{align*}
e(f)
&= 4\,\frac{|f'(z)|^2}{(|f(z)|^2 + 1)^2}
  - 4\,\frac{\lambda}{\eps}\,
    \Re\!\left(
        \frac{f'(z)}{(|f(z)|^2 + 1)^2}
    \right)
  + \frac{1}{2\eps^2}\,
    \frac{(|f(z)|^2 - 1)^2}{(|f(z)|^2 + 1)^2} \\[0.4em]
&= 4\,\frac{a^2}{(|z|^2 + a^2)^2}
   - 2\,\frac{\lambda a}{\eps}\,
     \frac{(\Re z)^2 - (\Im z)^2 + 2a\,\Re z + a^2}{(|z|^2 + a^2)^2}
   + \frac{2a^2}{\eps^2}\,
     \frac{(\Re z)^2}{(|z|^2 + a^2)^2}.
\end{align*}
Hence, the energy density admits a smooth extension at $z = -a$.  
In polar coordinates $z = r e^{i\theta}$, integration yields
\begin{align*}
\int_{D_R} e(f)\, \mathrm{d}z
&= 8\pi a^2 \int_0^R \frac{r\, \mathrm{d}r}{(r^2 + a^2)^2}
   - 4\pi \frac{\lambda a^3}{\eps}
     \int_0^R \frac{r\, \mathrm{d}r}{(r^2 + a^2)^2}
   + 2\pi \frac{a^2}{\eps^2}
     \int_0^R \frac{r^3\, \mathrm{d}r}{(r^2 + a^2)^2} \\[0.5em]
&= 4\pi \!\left(1 - \frac{\lambda a}{2\eps}\right)
     \frac{R^2}{R^2 + a^2}
   + \pi \frac{a^2}{\eps^2}
     \!\left(
       \ln\!\frac{R^2 + a^2}{a^2}
       - \frac{R^2}{R^2 + a^2}
     \right)\!,
\end{align*}
which establishes the stated formula.
\end{proof}

\subsubsection*{Cut-off argument}\label{subsec:cut-off}
 
To construct $f_{R,a}$ satisfying the prescribed boundary behaviour, the variable $z$ is expressed through the inverse stereographic projection $z = \Phi^{-1}(\Phi(z))$, followed by a cut-off modification analogous to that introduced in~\cite{doring2017compactness}.  

Define a smooth function $g : [0,\infty) \to \R$ such that
\begin{align}\label{eq:trial}
g(s) =
\begin{cases}
\ln(1 + s^2), & s \le 2, \\[0.3em]
\text{const}, & s \ge 4,
\end{cases}
\end{align}
and
\begin{align}\label{ineq:trial}
0 \le g'(s) \le \frac{2s}{s^2 + 1},
\qquad
0 \le -g''(s) \le \frac{C}{s^2 + 1}
\end{align}
for all $s \ge 2$.  

\medskip
Using $g$, define the modified stereographic map $\tilde\Phi : \R^2 \to \S^2$ by
\[
\tilde\Phi(x)
= \Bigl(
    g'(|x|)\,\frac{x}{|x|},
    \, \operatorname{sgn}(|x|^2 - 1)\, \sqrt{1 - g'(|x|)^2}
  \Bigr)^{\!T}.
\]
It follows that $\tilde\Phi(x) = \Phi(x)$ for $|x| \le 2$ and $\tilde\Phi(x) = \e{3}$ for $|x| \ge 4$.  
Consequently, $\Phi^{-1}\!\circ\! \tilde\Phi$ approximates the identity and diverges as $|x| \nearrow 4$.  
Recalling that $\Phi^{-1}(y, y_3) = \frac{y}{1 - y_3}$, one obtains
\[
\Phi^{-1} \circ \tilde\Phi(x)
= \frac{g'(|x|)}{1 - \operatorname{sgn}(|x|^2 - 1)\sqrt{1 - g'(|x|)^2}}\,
  \frac{x}{|x|}.
\]
Introducing the radial function
\[
r(s)
= \frac{g'(s)}{1 - \operatorname{sgn}(s^2 - 1)\sqrt{1 - g'(s)^2}},
\]
define its scaled version by $r_R(s) := R\,r(s/R)$, where $0 < R < \tfrac{1}{2}$.  
The modified ansatz is then given by
\[
f_{R,a}(z)
= \frac{r_R(|z|)\frac{z}{|z|} - a}
       {r_R(|z|)\frac{z}{|z|} + a}.
\]

By construction,
\[
f_{R,a}(z) = f(z) = \frac{z - a}{z + a}
\quad \text{for } |z| \le R,
\qquad
f_{R,a} \equiv 1
\quad \text{on } \C \setminus D_{2R},
\]
and $f_{R,a}$ is smooth on $\C \setminus \{-a\}$.  
The associated fields $u_{R,a}$ and $\m_{R,a} = \Phi(u_{R,a})$ contribute to the energy only on $D_{2R}$, with $\deg \m_{R,a} = -1$.  
According to Lemma~\ref{lemma:mRBounds} and Hölder’s inequality,
\[
\int_{D_{2R} \setminus D_R} e(\m_{R,a})\, \mathrm{d}x
\le C\!\left(\frac{a^2}{R^2} + \frac{a^2}{\eps^2}\right).
\]
Combining this with Lemma~\ref{lemma:localenergy} yields the upper bound
\begin{align}\label{eq:upper_aR_energy_bound}
4\pi\!\left(1 - \frac{\lambda a}{2\eps}\right)\!\frac{R^2}{R^2 + a^2}
+ \pi \frac{a^2}{\eps^2}
  \!\left(
    \ln\!\frac{R^2 + a^2}{a^2}
    - \frac{R^2}{R^2 + a^2}
  \right)
+ C\!\left(
    \frac{a^2}{R^2} + \frac{a^2}{\eps^2}
  \right)\!.
\end{align}
For $0<\eps <\eps_0$, define the rescaled map 
\begin{align*}
    \m_{\eps,R,a}(x) = \m_{R,a}\left(\frac{\eps_0}{\eps} x\right)
\end{align*}
and fix $R=1/2$. This reduces the estimate $\eqref{eq:upper_aR_energy_bound}$ to the case $\eps =\eps_0$. Then,
\begin{align*}
    I &\le 4 \pi\left(1 - \frac{\lambda a}{2\eps_0}\right)\!\frac{1}{1 + 4a^2}
+ \pi \frac{a^2}{\eps_0^2}
  \!\left(
    \ln\!\frac{1 + 4a^2}{4a^2}
    - \frac{1}{1 + 4a^2}
  \right)
+ C\!\biggl(
    1+\frac{1}{\eps_0^2}
  \biggr)a^2 \\
  &\le 4 \pi\left(1 - \frac{\lambda a}{2\eps_0}\right)
+ C_0\biggl(
    1+\frac{1}{\eps_0^2}
  \biggr)(a + a |\ln a|) a
\end{align*}
Hence, there exists $0<a<a_0(\eps_0,\lambda)$ such that
\begin{align*}
    a + a |\ln a|\le \frac{\lambda \,\eps_0}{4C_0(1+\eps_0^2)}.
\end{align*}
This proves the first statement of Proposition~\ref{prop:upper_bound}. 
To capture the asymptotic behaviour, consider the parametrization
\[
\frac{a}{\eps_0} = c_0 \frac{\lambda}{|\ln \lambda|},
\]
which is optimal for $c_0 = \tfrac{1}{2}$. As $\lambda \to 0$, this leads to
\[
E(\m_{R,a})
\le 4\pi\!\left(
    1 - \frac{\lambda^2}{8|\ln \lambda|}
  \right)
  + o\!\left(\frac{\lambda^2}{|\ln \lambda|}\right),
\]
thereby establishing Proposition~\ref{prop:upper_bound}.

\begin{proof}[Proof of Theorem~\ref{thm:existence}]
Fix $\eps_0>0$ and let $0<\eps<\eps_0$, $0<|\lambda|<1$. Assume $\lambda>0$; otherwise the change of sign $\m = (m, m_3) \mapsto \m^{-} = (-m, m_3)$ reduces the case $\lambda<0$ to the same setting. Let $(\m_k)_{k\in \N}\subset H^{1}_c(\D;\St)$ denote a minimizing sequence for $E=E_{\lambda,\eps}$ subject to the prescribed boundary and degree conditions.  
By \eqref{eq:toplowerbound}, the sequence $(\m_k)$ satisfies
\[
\m_k\rightharpoonup\m 
\quad\text{weakly in } H^{1}(\D;\S^2), 
\qquad
\m_k \to \m 
\quad\text{strongly in } L^{2}(\D;\S^2),
\]
and $\m_k(x)\to \m(x)$ almost everywhere in $\D$.  
Consequently, $|\m|=1$, and the Jacobians 
\(\omega(\m_k)=\m_k\cdot\partial_1\m_k\times\partial_2\m_k\)  
converge weakly-$\ast$ as measures on $\overline{\D}$.

Application of Theorem~E.1 in \cite{BrezisHarmonicmaps} together with Lemma~4.3 in \cite{LionsConcentrationCompactness} yields the existence of integers \(q_1,\ldots,q_N\in\Z\) and points \(x_1,\dots,x_N\in\D\) such that
\[
E_{\lambda,\eps}(\m) + 4\pi \sum_{k=1}^N |q_k| \le I,
\qquad 
\omega(\m_k) \rightharpoonup \omega(\m) + \sum_{k=1}^N q_k \delta_{x_k}.
\]
By Proposition~\ref{prop:upper_bound} and the non-negativity of the energy $E_{\lambda,\varepsilon}$, see e.g.~\eqref{eq:topo-lower-lambda}, all topological charges vanish, i.e.\ \(q_1 = \ldots = q_N = 0\).  
Hence \(\omega(\m_k) \rightharpoonup \omega(\m)\), and consequently \(\deg\m=-1\) and \(E_{\lambda,\eps}(\m)=I\).

Local minimality implies that \(\m\) satisfies weakly the Euler--Lagrange equation
\begin{equation} \label{eq:EL}
\bs{\tau}(\m) = \bs{f}(\m)  \quad \text{with} \quad \bs{\tau}(\m) = \Delta \m + |\nabla \m|^2 \m, 
\end{equation}
the tension field, and 
\begin{equation} \label{eq:EL_RHS}
\bs{f}(\m)= (\bs{1}- \m \otimes \m) \left( \frac{\lambda}{\eps}\mDMI+\frac{m_3}{\eps^{2}} \e 3 \right)
\end{equation}
which constitutes an $L^2$ perturbation of the harmonic map equation.  
A modification of Hélein’s regularity argument \cite{helein2002harmonic} (see below) then guarantees continuity and smoothness of~$\m$.  
\end{proof}

\begin{rem}
Adopting the notation from \cite{Riviere2007},
the Euler--Lagrange equation reads
\[
\Delta \m + \Omega \cdot \nabla \m  = \bs{f} \in L^2(\D; \R^3)
\]
for $\Omega = \m \otimes \nabla \m - \nabla \m \otimes \m$ and
$\bs{f}=\bs{f}(\m)$ as above. Taking into account that $\nabla \cdot \Omega
= \m \otimes \bs{f}  - \bs{f} \otimes \m \in L^2(\D;  \mathrm{so}(3))$
we obtain a Hodge decomposition $\Omega=\nabla \alpha +B$
for some $\alpha \in H^2 \cap H^1_0(\D; \mathrm{so}(3))$ and $B \in L^2(\D; \mathrm{so}(3) \otimes \R^2)$ with $\nabla \cdot B =0$. Decomposing $\m = \bs u + \bs v$ we
let $\bs v \in H^1_{(c,0)}(\D; \R^3)$ solve
\[
\Delta \bs v = \bs{f} - \nabla \alpha: \nabla \m \in L^p(\D; \R^3)
\quad \text{for all} \quad 1<p<2
\]
so that $v \in W^{2,p}(\D; \R^3) \subset C^{0,2-2/p}(\overline{\D}; \R^3)$ by the 
Calder\'on-Zygmund inequality and Morrey's embedding. 
After approximation, Poincar\'e's lemma
yields $B=\nabla^\perp \beta$ for some $\beta \in H^1(\D;  \mathrm{so}(3))$ with zero mean. Continuity of $\bs{u}$ then follows from Wente's inequality as in 
\cite{helein2002harmonic} Chapter 3.1. Then smoothness of $\m$ follows from the results on semilinear elliptic equations with quadratic growth in the gradient as established in \cite{LU}.
\end{rem}

\begin{rem}\label{rem:torus}
The unit disk $\D$ may be replaced by the flat two-dimensional torus
\[
\T^{2} = \R^{2}/\Z^{2},
\]
where Dirichlet boundary conditions are replaced by periodic ones.  
The reasoning in the proof remains valid under this setting.  
This torus formulation will be adopted in Section~\ref{sec:Energygap} and Section~\ref{sec:eps0} for the analysis of the large-domain limit.
\end{rem}

\section{The conformal limit} \label{sec:conformallimit}
Keeping the anisotropy constant $\varepsilon > 0$ fixed (for convenience, $\varepsilon = 1$), the limit $\lambda \to 0$ is referred to as the \emph{conformal limit}.  
This regime represents a manifestation of the bubbling phenomenon for Palais--Smale-type sequences associated with harmonic-type systems, as analyzed in~\cite{DingTian1995Approximateharmonicmap, qing1997bubbling, lin1998energy, Jost2019NoNeckDirichlet}.  
Let $\m_{\lambda}$ denote a family of local minimizers of $E_{\lambda} = E_{\lambda,1}$ provided by Theorem~\ref{thm:existence}.  
In the regime $\lambda \ll 1$, a bubbling analysis shows that the maps $\m_\lambda$ approach a degree-one harmonic profile.  
The subsequent analysis examines the asymptotic behavior of the total energy and its individual components in this limit.

\begin{lemma}\label{lemma:energyconformallimit}
We have 
\begin{align*}
 \lim_{\lambda \to 0} \int_{\D}  {|m_{3, \lambda}|^2}\, \mathrm{d}x =0 \quad \text{and} \quad 
    \lim_{\lambda \to 0} E_{\lambda}(\m_{\lambda})
    = \lim_{\lambda \to 0} 
      \int_{\D} \frac{|\nabla \m_{\lambda}|^2}{2}\, \mathrm{d}x
    = 4\pi.
\end{align*}
\end{lemma}

\begin{proof}
The lower and upper energy bounds established earlier imply that
\begin{align*}
    E_{\lambda}(\m_{\lambda})
    = 4\pi + \mathcal{O}(\lambda^2),
    \qquad
    \frac{1}{2} \int_{\D} |\nabla \m_{\lambda}|^2\, \mathrm{d}x
    = 4\pi + \mathcal{O}(\lambda^2),
\end{align*}
as $\lambda \to 0$. A scaling argument combined with Young's inequality yields
\begin{align} \label{eq:m3}
    \int_{\D} m_{3, \lambda}^2\, \mathrm{d}x
    &\le
    -\,\lambda
      \int_{\D} m_{3, \lambda}\, (\nabla \!\cdot\! m_{\lambda})\, \mathrm{d}x
      \le
    c\,\lambda^2
    + \frac{1}{2} \int_{\D} m_{3, \lambda}^2\, \mathrm{d}x,
\end{align}
which implies
\begin{align*}
    \int_{\D} m_{3, \lambda}^2\, \mathrm{d}x = \mathcal{O}(\lambda^2),
    \qquad
    \lambda
    \int_{\D} m_{3, \lambda}\, (\nabla \!\cdot\! m_{\lambda})\, \mathrm{d}x
    = \mathcal{O}(\lambda^2),
\end{align*}
as $\lambda \to 0$. Indeed, defining $\m_{\lambda,R}(x) := \m_{\lambda}(R^{-1}x)$ for $R<1$ and extending it constantly outside $D_R$, estimate~\eqref{eq:m3} follows directly from the fact that the one-sided derivative of $E_{\lambda}(\m_{\lambda,R})$ at $R=1$ is non-positive. This proves the claim.
\end{proof}

For critical points $\m$, the tension field (cf. \eqref{eq:EL}) satisfies
\[
\|\bs{\tau}(\m)\|_{L^2}^2 \lesssim \lambda^2 \|\nabla \m\|_{L^2}^2 + \|m_3\|_{L^2}^2.
\]

A standard small-energy regularity argument shows that there exists $\delta_0>0$ such that if $\m \in H^2(\D;\S^2)$ satisfies 
\[
\|\nabla \m\|_{L^2(D_R\cap \D)} \le \delta_0 \quad\text{and}\quad \bs{\tau}(\m) \in L^2,
\] 
then
\[
\int_{D_{R/2}\cap \D} |\nabla^2 \m|^2 \,\d{x} \lesssim \frac{1}{R^2} \int_{D_R\cap \D} |\nabla \m|^2 + |\bs{\tau}(\m)|^2 \,\d{x}.
\] 
Indeed, choosing a cutoff function $\eta \in C^\infty_c(D_R)$ with $\eta \equiv 1$ on $D_{R/2} \cap \D$ and $|\nabla \eta| \lesssim 1/R$, and using $\Delta \m = -|\nabla \m|^2 \m + \bs{\tau}(\m)$, we obtain
\[
\int_{D_R\cap \D} |\Delta \m|^2 \eta^2 \, \d{x} \lesssim \int_{D_R\cap \D} |\nabla \m|^4 \eta^2 + |\bs{\tau}(\m)|^2 \eta^2 \, \d{x},
\]
from which the estimate follows via Ladyzhenskaya's interpolation inequality and integration by parts.
Defining the blow-up scale
\[
R_\lambda := \sup \Big\{\rho>0 : \sup_{x \in \D} \int_{D_\rho(x) \cap \D} |\nabla \m_\lambda|^2 \, \d{x} \le \delta_0^2 \Big\},
\] 
concentration occurs:

\begin{lemma}
$R_\lambda \to 0$ as $\lambda \to 0$.
\end{lemma}

\begin{proof}
Indeed, if $R_\lambda$ did not converge to zero, then $\m_\lambda$ would remain regular at a uniform scale, and by the small-energy regularity estimate, we could extract a strong $H^1$ limit with constant Dirichlet boundary values, contradicting the nontrivial energy $\int_\D |\nabla \m_\lambda|^2 \sim 4\pi$. 
\end{proof}

Now let $\lambda_{k} \to 0$.
From the definition of $R_k=R_{\lambda_k}$, there exists a sequence $x_k=x_{\lambda_k} \in \D$ such that, for $\m_k=\m_{\lambda_k}$, 
\[
\int_{D_{R_k}(x_k)\cap \D} |\nabla \m_k|^2 \, \d{x} \ge \frac{\delta_0^2}{2}.
\] 
Bubbling near the boundary is possible only if (cf. \cite{Jost2019NoNeckDirichlet})
\[
\dist(x_k, \partial \D) \to 0 \quad \text{and}  \quad R_k / \dist(x_k, \partial \D) \to 0.
\]
Otherwise one would obtain a nontrivial smooth harmonic map on a half-plane with constant boundary, which is impossible. Moreover, bubbling analysis results in \cite{DingTian1995Approximateharmonicmap, qing1997bubbling, lin1998energy} 
guarantee energy decay and no-neck behavior in the intermediate region centered at $x_k$:
\begin{equation} \label{eq:energy_no_neck}
\lim_{\rho \to 0} \lim_{L \to \infty} \lim_{k \to \infty} \int_{(D_\rho \setminus D_{R_k L}) \cap \D} 
|\nabla \m_k|^2 \, \d{x} = 0
\end{equation}
and
\begin{equation} \label{eq:no_neck}
\lim_{\rho \to 0} \lim_{L \to \infty} \lim_{k \to \infty} 
\mathrm{osc}\big(\m_k; {(D_\rho \setminus D_{R_k L}) \cap \D }\big)  = 0.
\end{equation}

Consequently, in the blow-up coordinates $y \mapsto \m_k(x_k + R_k y)$, one can extract a nontrivial harmonic map $\h \in H^2_{\mathrm{loc}}(\R^2; \S^2)$ capturing the bubble, while the original maps $\m_{k}$ converge weakly to the constant map $\m_0=(c,0)$ away from the bubbling point. 
According to \eqref{eq:energy_no_neck} the energy decouples as
\[
\lim_{k \to \infty} \int_\D |\nabla \m_k|^2 \, \d{x} = \int_{\R^2} |\nabla \h|^2 \, \d{x} + \int_\D |\nabla \m_0|^2 \, \d{x},
\]
and the no-neck condition \eqref{eq:no_neck} ensures $\h(\infty) = (c,0)$.

\medskip

For minimizers $\m_{\lambda,\eps}$ obtained in Theorem~\ref{thm:existence} we obtain the following result:

\begin{thm}\label{thm:bubbling}

Fix $\eps > 0$ and let $\lambda_{k} \to 0$. 
Then there exist scales $R_{k} \to 0$ and points $x_{k} \in \D$ such that the rescaled maps
\[
\tm_{k}(x) = \m_{\lambda_{k},\eps}\bigl(R_{k}x + x_{k}\bigr)
\]
satisfy, up to a subsequence, 
\[
\tm_{k} \;\rightharpoonup\; \h \quad \text{in } H^{2}_{\rm loc}(\R^{2};\S^2),
\]
for some harmonic map $\h$ of unit degree with
\[
\h(\infty) = (c,0) \in \S^1 \times \{0\}.
\]
\end{thm}

\section{Energy gap for anisotropic harmonic maps}\label{sec:Energygap}
To exclude potential boundary effects, the discussion is restricted henceforth to maps 
$\m: \T^2 \to \S^2$ defined on the flat torus $\T^2$, where Theorem~\ref{thm:existence} remains applicable.  
The focus is on the large-domain regime $\eps \ll 1$, with $\lambda$ small but non-zero.  
This setting motivates the introduction of the effective \emph{easy-plane functional}
\begin{align}\label{eq:easyplane}
E_\lambda(m) 
= \frac{1}{2} \int_{\T^2} |\nabla m|^2 - \lambda^2 (\nabla \cdot m)^2  \, \mathrm{d}x,
\end{align}
defined for horizontal fields $m=(m_1,m_2)$.  
By an argument analogous to that in \eqref{eq:toplowerbound}, this functional provides an $\eps$-independent lower bound for the full energy,
\begin{equation}
E_\lambda(m) \le E_{\lambda,\eps}(\m),
\end{equation}
and thus constitutes a natural candidate for the $\Gamma$--limit of $E_{\lambda,\eps}$ as $\eps \to 0$.  
The precise relation between $E_{\lambda,\eps}$ and $E_\lambda$ will be established in Section~\ref{sec:eps0}.  

\medskip
The present section concerns the analysis of critical points $m$ of $E_{\lambda}$ in $H^1(\T^2;\S^1)$.  
Owing to the anisotropic dependence on the gradient, such maps will be referred to as \emph{anisotropic harmonic maps}.  
These maps do not arise from a simple change of metric.  
In analogy with the classical case, it will be shown that anisotropic harmonic maps possessing sufficiently small energy are necessarily trivial, thereby establishing an energy gap.
\medskip

By variation in the direction of the admissible tangent field $\psi=m^{\perp} \eta$ for $\eta \in H^1\cap L^{\infty}(\T^2)$, we obtain 
    \begin{align} \label{eq:weakangletest}
        0=\int_{\T^2}( m^{\perp} \cdot \nabla m) \cdot \nabla \eta + \lambda^2 (\nabla \cdot m)(\nabla \cdot m^{\perp}) \eta + \lambda^2 (\nabla \cdot m) (m^{\perp} \cdot \nabla) \eta \; \d x.
    \end{align}
Similarly to the harmonic map case, the current 
\[
m^{\perp} \cdot \nabla m = m_1 \nabla m_2 - m_2 \nabla m_1
\]
reduces to a phase gradient $\nabla \phi$ provided there exists a lift $m=e^{i \phi}$. The existence of such a lift follows from a topological argument once the energy of $m$ is sufficiently small.
\begin{lemma} \label{lemma:liftexistence}
    Let $m\in H^1(\T^2 ;\S^1) $ satisfy
    \begin{align*}
        \int_{\T^2} |\nabla m|^2 \d x < 4\pi^2.
    \end{align*}
    Then there exists $\phi\in H^1(\T^2)$ such that $m=e^{i \phi}$.
\end{lemma}
\begin{proof}
Without loss of generality, we may assume $m \in C^{\infty}(\T^2;\S^1)$ by a standard approximation argument at the critical exponent, see for instance \cite{SchoenUhlenbeckRegularity}, and the compactness of $\T^2$. For maps $\T^2 \simeq \S^1 \times \S^1 \to \S^1$, the homotopy classes $[\S^1\times \S^1;\S^1]\simeq\Z^2$ can be identified via $k=\deg\left.(m \right|_{\S^1\times \{y_0\}})$ and $\ell=\deg\left.(m \right|_{\{x_0\} \times \S^1})$ for some $x_0 ,y_0 \in \S^1$. 
Note that $m|_{\S^1\times \{y_0\}}$ is homotopic to $m|_{\S^1\times \{y\}}$ for any $y \in \S^1$, and similarly for the first component. Moreover,
    \begin{align*}
        \int_{\T^2} |\nabla m|^2 \,\d x 
        &= \int_{\T^2} |\partial_x m|^2 \,\d x + \int_{\T^2} |\partial_y m|^2 \,\d x \\
        &\geq 4 \pi^2 (k^2+\ell^2),
    \end{align*}
where we used the one-dimensional lower bound for the degree on each circle. In particular, $k=\ell=0$. Let us consider the $1$-form $\eta = -i m^{-1} \d{m}$. By the Hodge decomposition on $\T^2$, we can write
\begin{align*}
\eta = \d \phi + \d^{\ast} \psi + h,
\end{align*}
where $\phi$ is a smooth $0$-form, $\psi$ is a smooth $2$-form, and $h$ is a harmonic $1$-form. We may further assume that $\int\phi\, \d x  = 0$. Note that $h$ is constant and the corresponding function to $\psi$ is harmonic and therefore constant, since $m^{-1}\nabla m$ is curl-free, i.e. $\eta = \d \phi + h$. Using $\deg\left.(m \right|_{\S^1\times \{y_0\}})=\deg\left.(m \right|_{\{x_0\} \times \S^1})=0$, we obtain
    \begin{align*}
        \int_{\S^1\times \{y_0\}} \eta \d{s}= \int_{\{x_0\} \times \S^1} \eta \d{s}= 0.
    \end{align*}
    and hence $h =0$. Therefore $\eta = \d \phi $, which implies $\d{\, (e^{-i\phi} m)} = 0$ and in particular the claim.
\end{proof}
Substituting $m = e^{i \phi}$ and $\eta = \phi$ into \eqref{eq:weakangletest} yields
\begin{align} \label{ineq:boundingangle}
    \int_{\T^2} |\nabla \phi|^2 \, \mathrm{d}x
    \le c \lambda^2 \bigl(1 + \|\phi\|_{L^{\infty}}\bigr) \int_{\T^2} |\nabla \phi|^2 \, \mathrm{d}x.
\end{align}
To establish the energy gap, it remains to control $\phi$ pointwisely.  
It is natural to consider the function space 
\[
H^1(\T^2;\S^1) \cap H^3_{\mathrm{loc}}(\T^2 \setminus \{0\};\S^1),
\]
in which critical points satisfy the equation
\begin{align}\label{eq:m}
    \Delta m + |\nabla m|^2 m 
    - \lambda^2 \Bigl( \nabla(\nabla\cdot m) - \bigl[\nabla(\nabla\cdot m) \cdot m\bigr]\, m \Bigr) = 0
    \quad \text{in } \T^2 \setminus \{0\}.
\end{align}
The higher regularity away from the origin allows to follow the approach of Sacks and Uhlenbeck \cite{sacks1981existence}: first, small-energy regularity is established away from $0$; second, a Pohozaev identity is derived; finally, a growth estimate for the Dirichlet energy of $\phi$ is obtained.  
Since $|\nabla \phi| = |\nabla m|$, it suffices to perform this analysis directly in terms of $m$.  

\begin{lemma}[Small-energy regularity]\label{lemma:smallenergy}
Let $m \in H^1(\T^2;\S^1) \cap H^3_{\mathrm{loc}}(\T^2 \setminus \{0\};\S^1)$ be a critical point of $E_\lambda$.  
There exist constants $\lambda_0, \delta > 0$ such that, if $\|\nabla m\|_{L^2(\T^2)} \le \delta$ and $0 < |\lambda| < \lambda_0$, then for every $R>0$ with $D_{2R} \subset \T^2 \setminus \{0\}$,
\begin{align*}
R^2 \sup_{D_R} |\nabla m|^2 \;\le\; c \int_{D_{2R}} |\nabla m|^2 \, \mathrm{d}x.
\end{align*}
\end{lemma}

\begin{proof}
By a standard scaling argument, it suffices to consider $R = 1$.  
Let $\eta \in C_c^\infty(D_2)$ satisfy $\eta \equiv 1$ on $D_1$.  
Testing the Euler-Lagrange equation for $m$ with $(\Delta m)\, \eta^2$ yields
\begin{align*}
    \int_{D_2} |\Delta m|^2 \, \eta^2 \, \mathrm{d}x
    \le c \int_{D_2} |\nabla m|^4 \, \eta^2 \, \mathrm{d}x
      + c \lambda^2 \int_{D_2} |\nabla^2 m|^2 \, \eta^2 \, \mathrm{d}x.
\end{align*}
Applying Ladyzhenskaya's interpolation inequality,
\begin{align*}
    \int_{D_2} |\nabla m|^4 \, \eta^2 \, \mathrm{d}x
    \le c \left( \int_{D_2} |\nabla m|^2 \, \mathrm{d}x \right)
      \left( \int_{D_2} |\nabla^2 m|^2 \, \eta^2 \, \mathrm{d}x
           + \int_{D_2} |\nabla m|^2 |\nabla \eta|^2 \, \mathrm{d}x \right),
\end{align*}
and integrating by parts,
\begin{align*}
    \int_{D_2} |\nabla^2 m|^2 \, \eta^2 \, \mathrm{d}x
    \le c \int_{D_2} |\Delta m|^2 \, \eta^2 \, \mathrm{d}x
      + c \int_{D_2} |\nabla m|^2 |\nabla \eta|^2 \, \mathrm{d}x.
\end{align*}
Hence, for sufficiently small $\delta$ and $\lambda_0$,
\begin{align}\label{eq:H2estimate}
    \int_{D_2} |\nabla^2 m|^2 \, \eta^2 \, \mathrm{d}x
    \le c \int_{D_2} |\nabla m|^2 (\eta^2 + |\nabla \eta|^2) \, \mathrm{d}x.
\end{align}

Differentiating the Euler-Lagrange equation, testing with $\nabla(\Delta m)\, \eta^4$, and estimating similarly, one obtains
\begin{align*}
    \int_{D_2} |\nabla(\Delta m)|^2 \, \eta^4 \, \mathrm{d}x
    \lesssim \int_{D_2} \bigl(|\nabla m|^4 + |\nabla^2 m|^4 + |\nabla m|^6\bigr) \eta^4 \, \mathrm{d}x
          + \lambda^2 \int_{D_2} |\nabla^3 m|^2 \, \eta^4 \, \mathrm{d}x.
\end{align*}
Applying Ladyzhenskaya's interpolation inequality together with \eqref{eq:H2estimate} leads to
\begin{align*}
    \int_{D_2} |\nabla^3 m|^2 \, \eta^4 \, \mathrm{d}x
    \lesssim \int_{D_2} |\nabla m|^2 \, \mathrm{d}x.
\end{align*}
Finally, the continuous embedding $W^{2,1}(\R^2) \hookrightarrow C^0(\R^2)$ implies the uniform bound on $|\nabla m|$ in $D_1$.
\end{proof}

The result transfers to the annulus:

\begin{lemma}\label{lemma:pointwise estimate}
Let $m \in H^1(\T^2;\S^1) \cap H^3_{\mathrm{loc}}(\T^2 \setminus \{0\};\S^1)$ be a critical point of $E_{\lambda}$ with 
$\|\nabla m\|_{L^2(\T^2)} \le \delta$ and $0 < |\lambda| < \lambda_0$.  
Then there exists a constant $c>0$ such that
\begin{align*}
    |x|^2 \, |\nabla m(x)|^2 
    \le c \int_{D_{4|x|} \setminus D_{\frac{|x|}{4}}} |\nabla m|^2 \, \mathrm{d}y
\end{align*}
for every $x \in \T^2 \setminus \{0\}$.
\end{lemma}

\begin{proof}
By Morrey's embedding theorem, $m \in C^{1,\gamma}_{\mathrm{loc}}(\T^2 \setminus \{0\};\S^1)$ for any $\gamma \in (0,1)$.
The annulus $D_{2r} \setminus D_{r/2}$ can be covered by finitely many disks of radius $r/4$.  
Applying Lemma~\ref{lemma:smallenergy} on each disk and taking the supremum over the annulus yields
\begin{align*}
    r^2 \sup_{\partial D_r} |\nabla m|^2
    \le r^2 \sup_{D_{2r} \setminus D_{r/2}} |\nabla m|^2
    \le c \int_{D_{4r} \setminus D_{r/4}} |\nabla m|^2 \, \mathrm{d}x,
\end{align*}
which establishes the claim.
\end{proof}

\begin{prop}[Pohozaev identity]
Let $m \in H^1(\T^2;\S^1)\cap H^3_{\loc}(\T^2\setminus\{0\};\S^1)$ be a critical point of $E_{\lambda}$ with $\|\nabla m\|_{L^2} \le \delta$ and $0<|\lambda|<\lambda_0$. Then, for every $R>0$, 
 \begin{equation} \label{eq:Pohozaev_plane}
\int_{\del D_R} \left( \left| \dd{m}{\tau} \right|^2 -\left| \dd{m}{\nu} \right|^2+ \lambda^2\left(\nabla \cdot m \right)\left( \dd{m}{\nu} \cdot \nu- \dd{m}{\tau} \cdot \tau\right) \right) \, \d{s}=0. 
\end{equation}
\end{prop}

\begin{proof}
Consider the energy density
\[
e_{\lambda}(m) = \frac{1}{2} |\nabla m|^2 - \frac{\lambda^2}{2} (\nabla \cdot m)^2
\]
and the associated stress tensor
\[
T_{\alpha \beta}(m) = e_{\lambda}(m) \delta_{\alpha \beta} - \partial_{\alpha} m \cdot \partial_{\beta} m + \lambda^2 (\nabla \cdot m) \, \partial_{\beta} m_\alpha,
\]
which satisfies $\tr T(m) = 0$. Note that $T(m) \in W^{1,1}_{\mathrm{loc}}(\T^2 \setminus \{0\}; \R^{2 \times 2})$ and
\begin{align*}
\partial_{\alpha} T_{\alpha \beta}(m) 
&= \nabla m \cdot \nabla (\partial_\beta m) - \lambda^2 (\nabla \cdot m) \, \nabla \cdot (\partial_\beta m) 
- \Delta m \cdot \partial_\beta m \\
&\quad - \nabla m \cdot \nabla (\partial_\beta m) + \lambda^2 \nabla (\nabla \cdot m) \cdot \partial_\beta m + \lambda^2 (\nabla \cdot m) \, \nabla \cdot (\partial_\beta m) \\
&= - \Delta m \cdot \partial_\beta m + \lambda^2 \nabla (\nabla \cdot m) \cdot \partial_\beta m = 0,
\end{align*}
for $\beta = 1,2$ in $\T^2 \setminus \{0\}$, where the last equality follows from \eqref{eq:m} and the fact that $\partial_\beta m \perp m$.

Integrating the identity
\[
\partial_\alpha (T_{\alpha \beta}(m) x_\beta) = \tr T(m) = 0
\]
over the annulus $D_R \setminus D_r$ for $R > r > 0$ yields
\[
R \int_{\partial D_R} \nu_\alpha T_{\alpha \beta} \nu_\beta \, \d{s} - r \int_{\partial D_r} \nu_\alpha T_{\alpha \beta} \nu_\beta \, \d{s} = 0.
\]

By Lemma~\ref{lemma:pointwise estimate}, we have
\[
\left| r \int_{\partial D_r} \nu_\alpha T_{\alpha \beta} \nu_\beta \, \d{s} \right| \le c \, r^2 \sup_{|x|=r} |\nabla m|^2 \le c \int_{D_{4r}} |\nabla m|^2 \, \d{x} \to 0
\]
as $r \to 0$.

Finally, observing that
\[
\nu_\alpha T_{\alpha \beta} \nu_\beta = \frac{1}{2} |\nabla m|^2 - \frac{\lambda^2}{2} (\nabla \cdot m)^2 - |\partial_\nu m|^2 + \lambda^2 (\nabla \cdot m) (\partial_\nu m \cdot \nu),
\]
together with the decomposition
\[
|\nabla m|^2 = |\partial_\nu m|^2 + |\partial_\tau m|^2, \qquad 
\nabla \cdot m = (\partial_\nu m \cdot \nu) + (\partial_\tau m \cdot \tau),
\]
completes the proof.
\end{proof}

\begin{lemma}
Let $m \in H^1(\T^2;\S^1)\cap H^3_{\loc}(\T^2\setminus\{0\};\S^1)$ be a critical point of $E_{\lambda}$. 
Then there exist $\delta,\lambda_{0}>0$ and $\gamma \in (0,1)$ depending only on $\delta,\lambda_{0}$ such that for $\|\nabla m\|_{L^2}\le \delta$ and every $0<|\lambda|<\lambda_{0}$,
\begin{align*}
    \int_{D_r} |\nabla m|^{2}\,\d x \le  r^{\gamma} \int_{\mathbb{D}} |\nabla m|^2 \,\d x
\end{align*}
for all $0<r<1$.
\end{lemma}

\begin{proof}
The argument proceeds by constructing a radially symmetric harmonic replacement of $m$, following the strategy of Sacks and Uhlenbeck \cite{sacks1981existence}. For dyadic annuli 
\begin{align*}
A_k = A(2^{-k},2^{-k+1})=D_{2^{-k+1}} \setminus \overline{D_{2^{-k}}}
\end{align*}
we define a radially symmetric harmonic replacement $h$ by prescribing the mean values 
\begin{align*}
    h(2^{-k}) := \fint_{\partial D_{2^{-k}}} m \, \d{s}, 
    \qquad 
    \Delta h = 0 \text{ in } A_k, \quad h \text{ radial.}
\end{align*}
As in \cite{sacks1981existence} proof of Theorem 3.6, one obtains 
\begin{align}\label{eq:estimateh}
    \|m-h\|_{L^{\infty}(A_k)} \le c \|\nabla m\|_{L^{2}(\tilde{A}_k)} \le c\delta,
\end{align}
where $\tilde{A}_k = A(2^{-k-1},2^{-k+2})$.  
The $\dot{H}^1$-distance can be bounded via 
\begin{align*}
    \int_{\D} |\nabla(m - h)|^2 \, \d{x} &= \sum_{k=1}^{\infty} \int_{\partial A_k} (m - h) \cdot \partial_{\nu}m  \, \d{s} -  \sum_{k=1}^{\infty}\int_{A_k} (m-h) \cdot \Delta m \, \d{x} .
\end{align*}
For every $k\in \N$, by the definition of $h$ we have
\begin{align*}
    \int_{\partial D_{2^{-k}}} (m-h)\cdot h' \, \d{s} =\int_{\partial D_{2^{-k}}} (m-h) \, \d{s}  \cdot h' = 0.
\end{align*}
Therefore, we conclude 
\begin{align*}
  \sum_{k=1}^{\infty} \left.\int_{\partial D_{r}} (m - h) \cdot \partial_{\nu}m  \, \d{s} \right|_{r=2^{-k}}^{2^{-k+1}}= \int_{\S^1} (m - h) \cdot \partial_{\nu}m  \, \d{s},
\end{align*}
since by \eqref{eq:estimateh} and Lemma~\ref{lemma:pointwise estimate}
\begin{align*}
\left|\int_{\partial D_{2^{-N}}} (m - h) \cdot \partial_{\nu}m  \, \d{s} \right|&\le c2^{-N} \|m - h\|_{L^{\infty}(\partial D_{2^{-N}})} \|\nabla m\|_{L^{\infty}(\partial D_{2^{-N}})} \\
&\le c \int_{\tilde A_N} |\nabla m|^2 \d x \to 0
\end{align*}
for $N \to \infty.$ Using \eqref{eq:m}, we obtain 
\begin{align*}
    \int_{A_k} (m-h) \cdot \Delta m \, \d{x}  &= -\int_{A_k} (m-h) \cdot (|\nabla m|^2 m)\, \d{x} \\
    &+ \lambda^2 \int_{A_k} (m-h)\cdot \nabla(\nabla\cdot m) \d x\\
    &-\lambda^2 \int_{A_k} ((m-h)\cdot m )(\nabla(\nabla\cdot m)\cdot m ) \d x.
\end{align*}
With $\|m-h\|_{L^{\infty}} \le c \delta$, we have 
\begin{align*}
    \left|\int_{A_k} (m-h) \cdot (|\nabla m|^2 m)\, \d{x} \right| \le c \delta \int_{A_k} |\nabla m|^2 \d x.
\end{align*}
The remaining terms can be estimated via integration by parts 
\begin{align*}
    \int_{A_k} (m-h)\cdot \nabla(\nabla\cdot m) \d x &= \int_{\del A_k} ((m-h)\cdot \nu)(\nabla\cdot m) \d x\\
    &-\int_{A_k} \nabla\cdot(m-h)(\nabla\cdot m) \d x
\end{align*}
and therefore 
\begin{align*}
   \sum_{k=1}^{\infty} \int_{A_k} (m-h)\cdot \nabla(\nabla\cdot m) \d x &\le \int_{\S^1} |m-h| |\partial_{\nu} m| \d{s}\\
   &+ \frac{1}{2}\int_{\D} |\nabla(m-h)|^2 + |\nabla m|^2 \d{x}.
\end{align*}
The nonlinear term is analogous. Therefore, we have 
\begin{align*}
    \int_{\D}|\nabla (m-h)|^2 \d{x} &\le  (1+c\lambda^2)\left(\int_{\S^1} |m-h|^2\d{s}\right)^{\frac12}\left(\int_{\S^1} |\partial_{\nu} m|^2\d{s}\right)^{\frac12}\\
    &+ c(\delta+\lambda^2) \int_{\D}|\nabla m|^2 \d x.
\end{align*}
Using the Poincaré inequality on $\mathbb{S}^1$ and Young's inequality, we infer
\begin{align*}
    \left(\int_{\S^1} |m-h|^2\d{s}\right)^{\frac12}\left(\int_{\S^1} |\partial_{\nu} m|^2\d{s}\right)^{\frac12} \le \frac{1}{2} \int_{\S^1} |\nabla m|^2 \d{s} .
\end{align*}
Now, $\partial_\tau h\equiv0$, together with Pohozaev's identity \eqref{eq:Pohozaev_plane} for $E_\lambda$, implies
\begin{align*}
    \int_{\D} |\nabla m|^2 \d x \le (2 + c \lambda^2) \int_{\D} |\partial_{\tau} m|^2 \d x  \le (2 + c \lambda^2) \int_{\D} |\nabla (m-h)|^2 \d x.
\end{align*}
Therefore, in particular,
\begin{align*}
    (1 - c(\delta+\lambda^2)) \int_{\mathbb{D}} |\nabla m|^2 \d x
    \le (1 + c\lambda^2) \int_{\mathbb{S}^1} |\nabla m|^2 \d{s}.
\end{align*}
Setting 
\begin{align*}
    \gamma := \frac{1 - c(\delta+\lambda_0^2)}{1 + c\lambda_0^2} \in (0,1),
\end{align*}
and applying a standard scaling argument, we obtain
\begin{align*}
    \gamma \int_{D_r} |\nabla m|^2 \d x \le r \int_{\partial D_r} |\nabla m|^2 \d{s} ,
\end{align*}
which integrates to the desired inequality.
\end{proof}

By Morrey’s Dirichlet growth lemma for $\phi$ we conclude that $\phi\in C^{0,\frac{\gamma}{2}}$ with, in particular,
\begin{align*}
    \|\phi-\phi(0)\|_{L^{\infty}}^2 \le c \int_{\D} |\nabla m|^2 \d{x} \le c \delta^2
\end{align*}
for some constant $c>0$. For $\delta,\lambda_0>0$ sufficiently small, \eqref{ineq:boundingangle} forces $\nabla\phi\equiv0$, hence $\phi\equiv\phi_0$. This proves Theorem~\ref{thm:energygap}.

\begin{thm}\label{thm:energygap}
Let $m \in H^1(\T^2;\S^1)\cap H^3_{\loc}(\T^2\setminus\{0\};\S^1)$ be a critical point of $E_{\lambda}$. There exist $\lambda_0,\delta>0$ such that, if $\|\nabla m\|_{L^2} \le \delta$ and $0<|\lambda|<\lambda_0$, then $m\equiv c$ for some constant $c \in\S^1$.
\end{thm}

\section{The large-domain limit}\label{sec:eps0}
We fix $\lambda$ and consider the limits of $\m_{\lambda,\eps}$, minimizer from Theorem~\ref{thm:existence} with $\D$ replaced by $\T^2$, 
as $\varepsilon \to 0$, obtaining bimeron configurations on $\R^2$.

\begin{thm} \label{thm:R2existence}
Let $0 < |\lambda| < \lambda_0 \ll 1$ be fixed, and let $\eps_{k} \to 0$ as $k \to \infty$.
There exist scales $R_{k} \to 0$ and points $x_{k} \in \T^2$ such that the rescaled maps
\[
\tm_{k}(x) = \m_{\lambda,\eps_k}\bigl(R_{k}x + x_{k}\bigr)
\]
satisfy, up to a subsequence, 
\[
\tm_{k} \;\rightharpoonup\; \m \quad \text{in } H^{2}_{\rm loc}(\R^{2};\S^2),
\]
and, for some $\eps_0 > 0$, the limit $\m \in H^1(\R^2;\S^2)$ is a unit-degree critical point of
\[
\int_{\R^2} 
\frac{1}{2}|\nabla \m|^2 
+ \lambda \, (\nabla \cdot m)\, \frac{m_3}{\eps_0}  
+ \frac{1}{2} \left( \frac{m_3}{\eps_0} \right)^2 \, \mathrm{d}x.
\]
\end{thm}

\subsection{Small energy regularity} \label{subsec:higherorder}
We adopt the notation $e_{\varepsilon}(\m) = e_{0,\varepsilon}(\m)$, where
\begin{align*}
e_{\varepsilon}(\m)
    = \frac{1}{2}\left( |\nabla \m|^2 + \left(\frac{m_3}{\varepsilon}\right)^2 \right),
\end{align*}
so that, including the DMI term \eqref{eq:DMI}, the density of the full energy is
\begin{align*}
e_{\lambda,\varepsilon}(\m)
    = e_{\varepsilon}(\m) + \frac{\lambda}{\varepsilon}\, e_{\mathrm{DMI}}(\m).
\end{align*}

\begin{prop} \label{prop:smallenergyregularity}
There exist positive constants $\lambda_0, \delta_0,$ and $C$ such that if 
\begin{align*}
 \int_{D_R} e_\eps(\m)  \,  \d x \le \delta_0^2 
\end{align*}
for some $R>0$ and $\m$ is a critical point of $E_{\lambda,\eps}$ with $0<|\lambda|<\lambda_0$, then for $k=1,2$
\begin{align*}
R^{2k} \int_{D_{R/4}} \left|\nabla^{k+1} \m \right|^2 + \frac{|\nabla^{k} m_3|^2}{\eps^2}  \,  \d x
\le C \int_{D_R} e_\eps(\m) \,  \d x,
\end{align*}
and in particular
\begin{align*}
\sup_{D_{R/4}} e_\eps(\m) \le C \fint_{D_R} e_\eps(\m) \,  \d x.
\end{align*}
\end{prop}

The bounds are independent of $\eps$, and by scaling, they are also independent of $R$. The proof strategy follows the standard argument based on Schoen’s trick \cite{schoen1984analytic}, starting from higher-order energy estimates of the form
\begin{equation} \label{eq:iteration}
\int_{D_{R/4}} e_\eps(\nabla^2 \m)  \,  \d x \lesssim \int_{D_{R/2}} e_\eps(\nabla \m)  \,  \d x + \int_{D_{R/2}} e_\eps(\m)  \,  \d x \lesssim \int_{D_R} e_\eps(\m) \,  \d x,
\end{equation}
with constants depending only on $R$ and $\sup_{D_R} e_\eps(\m)$. These estimates are
obtained by testing the Euler-Lagrange equation 
\eqref{eq:EL} with $\eta^2 \Delta \m$ and $\nabla \cdot (\eta^2 \nabla \Delta \m)$,
respectively, where $\eta$ is a smooth cut-off function. 
The exchange and anisotropy terms are handled via integration by parts together with Young’s inequality. The key is to absorb the highest--order terms in the perturbative bounds for the DMI contributions in 
\eqref{eq:EL_RHS}
\begin{align*}
\bs{h}(\m)= (\bs{1}-\m \otimes \m) \binom{-\nabla m_3}{\nabla \cdot m}
= \binom{-\nabla m_3-2 e_{\rm DMI}(\m) m}{\nabla \cdot m -2 e_{\rm DMI}(\m) m_3},
\end{align*}
which is possible for small enough $|\lambda|$ as the following lemma shows:

\begin{lemma} 
For any smooth $\m$ and compactly supported $\eta$,
\begin{align*}
\frac{1}{\eps} \left|  \int_{D_R}  \bs{h}(\m) \, \Delta \m \, \eta^2 \,\d{x} \right| 
&\lesssim \int_{D_R} \left( |\nabla^2 m|^2 + |\nabla m|^4 +
\frac{1}{\eps^2} |\nabla m_3|^2 \right) \eta^2 \,\d{x} 
\\ &+ \int_{D_R} \left( |\nabla m|^2 + \frac{1}{\eps^2} m_3^2 \right) |\nabla \eta|^2 \,\d{x} 
\end{align*}
and
\begin{align*}
\frac{1}{\eps} \left|  \int_{D_R} \nabla \bs{h}(\m) \cdot \nabla  \Delta  \m \, \eta^2 \,\d{x} \right| 
&\lesssim \int_{D_R} \left( |\nabla \Delta m|^2 + \frac{|\nabla^2 m_3|^2}{\eps^2} \right) \eta^2 \,\d{x} \\
 & + \int_{D_R} \left( \frac{m_3}{\eps} \right)^2 \left(|\nabla \m|^4 + |\nabla^2 \m|^2\right) \eta^2 \,\d{x}  \\
 & + \int_{D_R} |\nabla^2 m|^2 |\nabla \eta|^2 \,\d{x} .
\end{align*}
\end{lemma}

\begin{proof}
Integration by parts gives
\begin{align*}
\int_{D_R}  \bs{h}(\m) \, \Delta \m \, \eta^2 \,\d{x} 
&= - \int_{D_R} \nabla m_3 \, (\Delta m + \nabla \nabla \cdot m) \eta^2    \,  \d x
- 2 \int_{D_R} (\nabla m_3 \cdot \nabla \eta) (\nabla \cdot m) \,  \eta  \,  \d x\\
&\quad + 2 \int_{D_R} m_3 (\nabla \cdot m) |\nabla m|^2 \eta^2  \,  \d x + \int_{D_R} m_3 (m \cdot \nabla) \big(|\nabla m|^2 \eta^2 \big)  \,  \d x.
\end{align*}
Young's inequality implies the first claim. The second claim follows similarly.
\end{proof}

\begin{proof}[Proof of Proposition~\ref{prop:smallenergyregularity}]
Using the local $H^2$-bounds for $\nabla \m$ and for $m_3/\varepsilon$ provided by \eqref{eq:iteration}, we deduce a pointwise bound for $e_\varepsilon(\m)$ and, in turn, obtain a homogeneous version of the $H^2$-bounds as claimed. By rescaling $x$ and $\varepsilon$, it suffices to consider the case $R = 1$.

Let $\m$ be a critical point of $E_{\lambda, \eps}$ with $|\lambda| < \lambda_0$ and
\[
\int_{D_1} e_\eps(\m) \, \d{x} \le \delta_0^2,
\]
for some $\delta_0$ to be chosen later. 
By continuity there exists $\rho \in [0,1)$ such that
\[
(1-\rho)^2 \sup_{D_\rho} e_\eps(\m) = \max_{\sigma \in [0,1]} \left[(1-\sigma)^2 \sup_{D_\sigma} e_\eps(\m)\right].
\]
We define $e_0 = \sup_{D_\rho} e_\eps(\m)$, which is attained at some $x_0 \in \overline{D_\rho}$, and claim that for sufficiently small $\delta_0$, we have
\[
(1-\rho)^2 e_0 \le 1.
\]
If this does not hold, we proceed by defining the rescaled map
\[
\tm(x) = \m\left(x_0 + \frac{x}{\sqrt{e_0}}\right), \quad e_{\sqrt{e_0} \eps}(\tm)(x) = \frac{1}{e_0} e_\eps(\m) \left(x_0 + \frac{x}{\sqrt{e_0}}\right).
\]
Then, we have $e_{\sqrt{e_0} \eps}(\tm)(0) = 1$ and $\sup_{D_{1/2}} e_{\sqrt{e_0} \eps}(\tm) \le 4$
as in \cite{schoen1984analytic, doring2017compactness}.

Now, \eqref{eq:iteration} applies to $\tm$ with $R = 1/2$ and the modified $\eps$, and Sobolev embedding $H^2(D_{1/8}) \hookrightarrow L^\infty(D_{1/8})$ implies
\[
1 = e_{\sqrt{e_0} \eps}(\tm)(0) \le c \int_{D_1} e_\eps(\m) \, \d{x} \le c \delta_0^2,
\]
which is a contradiction for sufficiently small $\delta_0$.
\end{proof}

\subsection{The far field}\label{subsec:farfield}
Let $\m_{\lambda,\eps}$ be a minimizer of 
\[
E_{\lambda,\eps}(\m) = \int_{\T^2} e_{\lambda,\eps}(\m) \, \mathrm{d}x
\]
in the topological class $\deg=-1$. Consider a sequence $\eps_k \to 0$ and set $\m_k := \m_{\lambda,\eps_k}$ such that
\[
\m_k \rightharpoonup \m \quad \text{weakly in } H^1(\T^2;\S^2).
\]
Note that the out-of-plane component vanishes in the limit,
\[
\int_{\T^2} m_{3,k}^2 \le c \eps_k^2 \to 0 \quad \text{as } k\to\infty,
\]
and the Dirichlet energy is uniformly bounded,
\[
\int_{\T^2} |\nabla \m_k|^2 \le 4\pi(1 + c \lambda^2) < 8\pi.
\]

By the same argument as in Theorem~\ref{thm:existence} (see Theorem E.1 in \cite{BrezisHarmonicmaps} and Lemma 4.3 in \cite{LionsConcentrationCompactness}), we have, weakly-$\ast$ as measures,
\[
\frac12 |\nabla \m_k|^2 \rightharpoonup \mu \ge \frac12 |\nabla \m|^2 + 4\pi \delta_0, \qquad
\omega(\m_k) \rightharpoonup - 4\pi \delta_0
\]
for some measure $\mu$. Therefore, for any fixed $\rho > 0$,
\[
\limsup_{k \to \infty} \frac12 \int_{\T^2 \setminus D_\rho} |\nabla \m_k|^2 \, \mathrm{d}x \le c \lambda^2.
\]

Write $\m_k = (m_k, m_{3,k})$, $\m = (m,0)$, and define $v_k := m_{3,k}/\eps_k \in L^2(\T^2)$. Using Proposition~\ref{prop:smallenergyregularity} and covering $\T^2 \setminus D_\rho$ by disks of radius $\rho/2$, we obtain, for $j=1,2$,
\[
\rho^{2j} \int_{\T^2 \setminus D_\rho} |\nabla^{j+1} \m_k|^2 + |\nabla^j v_k|^2 \, \mathrm{d}x
\le C \int_{\T^2 \setminus D_{\rho/2}} |\nabla \m_k|^2 + v_k^2 \, \mathrm{d}x.
\]
Thus, up to a subsequence,
\[
m_k \rightharpoonup m \quad \text{in } H^1(\T^2;\S^1) \cap H^3_{\rm loc}(\T^2 \setminus \{0\};\S^1), \qquad
v_k \rightharpoonup v \quad \text{in } L^2(\T^2) \cap H^2_{\rm loc}(\T^2 \setminus \{0\}).
\]

Testing the Euler-Lagrange equation for $\m_k$ with
\[
\boldsymbol{\psi} = \begin{pmatrix} - m_k^\perp \\ 0 \end{pmatrix} \eta, \quad \eta \in C_c^\infty(\T^2 \setminus D_\rho),
\]
and passing to the limit $k\to\infty$ yields
\[
\int_{\T^2} (\nabla m \cdot m^\perp) \cdot \nabla \eta - \lambda (\nabla v \cdot m^\perp) \eta \, \mathrm{d}x = 0.
\]
Hence, $m \in H^3(\T^2 \setminus D_\rho;\S^1)$ solves
\begin{equation}\label{eq:(m,v)}
(\Delta m + \lambda \nabla v) \cdot m^\perp = 0 \quad \text{in } \T^2 \setminus D_\rho.
\end{equation}

Similarly, testing with
\[
\boldsymbol{\psi} = \eps_k \eta \, \e 3, \quad \eta \in C_c^\infty(\T^2 \setminus D_\rho),
\]
and passing to the limit $k \to \infty$ shows that $v$ is determined by $m$, leading to a coupled system. Therefore, $m$ solves \eqref{eq:m} in $\mathcal{D}'(\T^2 \setminus \{0\})$.  

To extend $m$ to a weak solution on all of $\T^2$, we use a logarithmic cut-off function to approximate test functions away from the singular point \cite{GuangLiWangZhouMinMax,HirschLammIndexHarmonicMaps,ChoiSchoenMinEmbedding}:

\begin{lemma}\label{lemma:extension}
Let $m \in H^1(\T^2;\S^1) \cap H^2_{\rm loc}(\T^2 \setminus \{0\})$ be a solution to \eqref{eq:m} in $\mathcal{D}'(\T^2 \setminus \{0\})$. Then $m$ solves \eqref{eq:m} in $\mathcal{D}'(\T^2)$.
\end{lemma}

\begin{proof}
Let $\psi \in H^1 \cap L^\infty(\T^2;\R^2)$ and let $\eta_\rho$ be a logarithmic cut-off function with $\eta_\rho \to 1$ pointwise and $\|\nabla \eta_\rho\|_{L^2} \to 0$ as $\rho \to 0$. Testing \eqref{eq:m} with $\eta_\rho \psi$ and taking the limit $\rho \to 0$ gives the result.
\end{proof}

By Theorem~\ref{thm:energygap} and Proposition~\ref{prop:smallenergyregularity}, we conclude:

\begin{prop}
For any sequence $(\m_k)$ of minimizers of $E_{\lambda,\eps_k}$ in the topological class $\deg=-1$ with $\eps_k \to 0$, there exists $c \in \S^1$ such that, up to a subsequence, for every $\rho>0$,
\[
\lim_{k \to \infty} \|\m_k - (c,0)\|_{L^\infty(\T^2 \setminus D_\rho)} = 0.
\]
\end{prop}

\subsection{Asymptotic scaling of the core radius}
Let $\m_{\lambda,\eps}$ be a minimizer of 
\begin{align*}
    E_{\lambda,\eps} (\m) = \int_{\T^2} e_{\lambda,\eps}(\m) \, \mathrm{d}x
\end{align*}
in the topological class $\deg=-1$. As before, we define the core radius
\begin{align*}
    R_{\lambda,\eps} = \sup \Bigl\{ \rho>0 : 
    \sup_{x \in \T^2} \int_{D_\rho(x)} e_\eps(\m_{\lambda,\eps}) \, \mathrm{d}x \le \delta_0^2 \Bigr\}
\end{align*}
where $\delta_0$ is taken from Proposition~\ref{prop:smallenergyregularity}.  
We track points $x_{\lambda,\eps}$ such that
\begin{align} \label{eq:good_points}
    \frac{\delta_0^2}{2} \le \int_{D_{R_{\lambda,\eps}}(x_{\lambda,\eps})} e_\eps(\m_{\lambda,\eps}) \, \mathrm{d}x \le \delta_0^2.
\end{align}
Without loss of generality, assume $x_{\lambda,\eps} = 0$ and define the blow-up map
\begin{align} \label{eq:blow_up_map}
    \tm_{\lambda,\eps}(x) := \m_{\lambda,\eps}(R_{\lambda,\eps} x).
\end{align}

\begin{prop}\label{prop:notzero}
There exists $\lambda_0>0$ such that, for every $0<|\lambda|<\lambda_0$, 
\begin{align*}
    \liminf_{\eps \to 0} \frac{R_{\lambda,\eps}}{\eps} > 0.
\end{align*}
\end{prop}

\begin{proof}
Suppose there exists a sequence $\eps_k \to 0$ such that
for $R_k := R_{\lambda,\eps_k}$
\begin{align*}
    \frac{R_k}{\eps_k} \to 0.
\end{align*}
Let $\m_k := \m_{\lambda,\eps_k}$ and $\tm_k$ the corresponding blow-up map \eqref{eq:blow_up_map}. Then $\tm_k$ is bounded in $H^2_{\rm loc}(\R^2;\S^2)$, and $R_k \to 0$. Up to a subsequence, there exists a limit map $\tm \in H^2(\R^2;\S^2)$ with $\tm_k \rightharpoonup \tm$ weakly in $H^2_{\rm loc}(\R^2;\S^2)$ and strongly in $H^1_{\rm loc}(\R^2;\S^2)$.
The rescaled equation \eqref{eq:EL} reads
\begin{align*}
\Delta \tm_k + |\nabla \tm_k|^2 \tm_k
+ \lambda \frac{R_k}{\eps_k} \Bigl( \begin{pmatrix} - \nabla \tms_{3,k} \\ \nabla \cdot \tms_k \end{pmatrix} 
- \big(\begin{pmatrix} - \nabla \tms_{3,k} \\ \nabla \cdot \tms_k \end{pmatrix} \cdot \tm_k\big) \tm_k \Bigr) \\
+ \left( \frac{R_k}{\eps_k} \right)^2 (\tms_{3,k} \e{3} - \tms_{3,k}^2 \tm_k) = 0.
\end{align*}
By the uniform $H^2$ bounds, the limit $\tm: \R^2 \to \S^2$ is harmonic with Dirichlet energy 
\begin{align*}
%\int_{\R^2} \frac{|\nabla \tm|^2}{2} \, \mathrm{d}x = 
\lim_{L \to \infty} \lim_{k \to \infty} \int_{D_L} e_{\lambda,\frac{\eps_k}{R_k}}(\tm_k) \, \mathrm{d}x \le \limsup_{k \to \infty}\int_{\T^2} e_{\lambda, \eps_k}(\m_k) 
\le 4 \pi \left(1- c(\lambda)\right)
\end{align*}
where we used the non-negativity of the energy density and Proposition~\ref{prop:upper_bound} which 
holds true on a torus.
By the energy gap property (cf. e.g. \cite{Brezis1983harmonic}), $\tm$ must be constant. This contradicts
\begin{align*}
\frac{1}{2} \int_{D_1(0)} |\nabla \tm|^2 \, \mathrm{d}x= 
\lim_{k \to \infty} \int_{D_{R_k}(0)} e_\eps(\m_k) \, \mathrm{d}x
\ge \frac{\delta_0^2}{2}
\end{align*}
proving the claim.
\end{proof}

\begin{prop}\label{prop:notinfty}
There exists $\lambda_0>0$ such that, for every $0<|\lambda|<\lambda_0$, 
\begin{align*}
    \limsup_{\eps \to 0} \frac{R_{\lambda,\eps}}{\eps} < \infty.
\end{align*}
\end{prop}

\begin{proof}
Assume by contradiction that the claim fails. Then there exist sequences $(\lambda_j)_{j\in\N}$ and $(\eps_k)_{k\in\N}$ with $\lambda_j \to 0$ and $\eps_k \to 0$ such that
\begin{align}\label{eq:contradictionassumption}
\lim_{k \to \infty} \frac{R_{\lambda_j,\eps_k}}{\eps_k} = \infty 
\qquad \text{for all } j \in \N.
\end{align}
We set $R_{j,k} := R_{\lambda_j,\eps_k}$ and let $\m_{j,k}$ be the corresponding minimizers. 
We first claim that $\lim_{k \to \infty} R_{j,k} = 0$ for every $j \in \N$. For otherwise
Proposition~\ref{prop:smallenergyregularity} implies strong $H^1(\T^2)$ compactness of
$(\m_{j,k})_{k \in \N}$ with limits $\m_{j}=(m_j,0)$ contradicting $\deg \m_{j,k}=-1$
for every $k \in \N$. Hence we consider blow-up sequences 
\[
\bs{u}_{j,k}(x):=\m_{j,k}(R_{j,k}x)
\]
satisfying
\begin{equation} \label{eq:res_smallness}
\frac{1}{2} \int_{D_1(x)}  |\nabla \bs{u}_{j,k}|^2 + |v_{j,k}|^2 \, \d{y} \le \delta_0^2
\end{equation}
for all $x \in \R^2$ and 
\begin{equation} \label{eq:res_lower}
\frac{1}{2} \int_{D_1(0)} |\nabla \bs{u}_{j,k}|^2 + |v_{j,k}|^2 \, \d{x}
\ge  \frac{\delta_0^2}{2} 
\end{equation}
where
\[
v_{j,k}:= \frac{u_{3,j,k}}{\eps_{j,k}}  \quad \text{and} \quad
\eps_{j,k} := \frac{\eps_k}{R_{j,k}} \to 0 \quad \text{as} \quad k \to \infty.
\]
In view of \eqref{eq:res_smallness},
Proposition~\ref{prop:smallenergyregularity} implies strong subconvergence
\[
\bs{u}_{j,k} \to \bs{u}_{j}=(u_j,0) 
\quad \text{and} \quad
v_{j,k} \to v_j 
\]
in $H^2_{\rm loc}$ and $H^1_{\rm loc}$, respectively, such that 
\[
\int_{D_{1/2}(x)} |\nabla^2 \bs{u}_j|^2 + |\nabla v_j|^2\,\d{y}
\le C
\int_{D_1(x)} |\nabla \bs{u}_j|^2 + v_j^2 \,\d{y} .
\]
for every $x \in \R^2$. Moreover, 
\begin{equation} \label{eq:EL_u_v}
-\nabla \cdot \left(\nabla u_j \cdot u_j^\perp \right) = \lambda_j 
 \nabla v_j\cdot u_j^\perp. 
\end{equation}
Note that by \eqref{eq:toplowerbound} and Proposition~\ref{prop:upper_bound}
\[
\frac{1}{2} \int_{\R^2} |\nabla u_j|^2 +v_j^2\, \d{x} \le \limsup_{k\to \infty} \int_{\T^2}
e_{\eps_k}(\m_{j,k}) \, \d{x} \le \frac{4\pi}{1-|\lambda_j|}.
\]
Moreover, a scaling argument (as in the derivation of \eqref{eq:m3}) yields
\begin{equation} \label{eq:v_bound}
\int_{\R^2} v_j^2 \, \d{x} \le c \lambda_j^2 \int_{\R^2} (\nabla \cdot u_j)^2 \d{x}
\to 0 \quad \text{as $j \to \infty$.}
\end{equation}
Passing to the limit in \eqref{eq:EL_u_v} we obtain a 
harmonic map $u:\R^2 \to \S^1$ of finite energy, hence $u$ is constant. The local strong $H^1$ convergence and \eqref{eq:v_bound} contradict  \eqref{eq:res_lower}.
\end{proof}

\subsection{Energy decay in the neck domain}
\begin{lemma}\label{lemma:neckharmonic}
Let $(\bs{\Psi}_k)_{k \in \N}$ be a sequence of harmonic maps with $\deg \bs{\Psi}_k = -1$ given by
\[
\bs{\Psi}_k(z) = \Phi\!\left( \frac{z + z_k - a_k}{z + z_k + a_k} \right),
\]
where $z_k \in \C$ and $a_k>0$.  
Assume that for some sequence $R_k \to 0$,
\[
\lim_{k \to \infty} \frac{|z_k|}{R_k} = s \in [0,\infty),
\qquad
\lim_{k \to \infty} \frac{a_k}{R_k} = a \in (0,\infty).
\]
Then, for every fixed $\rho>0$,
\[
\lim_{L \to \infty} \lim_{k \to \infty}
\int_{D_\rho \setminus D_{R_k L}} \frac{1}{2} |\nabla \bs{\Psi}_k|^2 \, \d z = 0.
\]
\end{lemma}

\begin{proof}
Define $\bs{\Psi}_k^0 = \Phi\!\left(\frac{z - a_k}{z + a_k}\right)$.  
A direct computation yields
\[
\frac{1}{2} \int_{D_r(z_k)} |\nabla \bs{\Psi}_k^0|^2 \, \d z
= \int_{D_r(z_k)} \frac{4 a_k^2}{(|z|^2 + a_k^2)^2} \, \d z.
\]
Setting $a=a_k$ and $z_0=z_k$, consider the integral
\[
I(r,z_0) = \int_{D_r(z_0)} \frac{4a^2}{(|z|^2 + a^2)^2} \, \d z.
\]
Since $\Delta \ln(|z|^2 + a^2) = \frac{4a^2}{(|z|^2 + a^2)^2}$, the divergence theorem gives
\[
I(r,z_0) 
= \int_{\partial D_r(z_0)} \frac{2z}{|z|^2 + a^2}\cdot \frac{z - z_0}{r}  \, \d{s}
= r\int_{-\pi}^{\pi} \frac{2|z_0|\cos \theta + 2r}{r^2 + |z_0|^2 + 2r|z_0|\cos\theta + a^2} \, \d \theta.
\]
Using the identity
\[
\int_{-\pi}^{\pi} \frac{\d \theta}{A + B \cos\theta}
= \frac{2\pi}{\sqrt{A^2 - B^2}}, \quad A > |B|,
\]
with $A = r^2 + |z_0|^2 + a^2$ and $B = 2r|z_0|$, one obtains
\[
I(r,z_0)
= 2\pi\!\left(1 + \frac{r^2 - |z_0|^2 - a^2}
{\sqrt{(a^2 + (r - |z_0|)^2)(a^2 + (r + |z_0|)^2)}}\right).
\]
Hence,
\[
\lim_{k \to \infty}\frac{1}{2} \int_{D_\rho } |\nabla \bs{\Psi}_k|^2 \, \d z = 4 \pi,
\]
and
\[
\lim_{k \to \infty}\frac{1}{2} \int_{ D_{R_k L}} |\nabla \bs{\Psi}_k|^2 \, \d z
= 2\pi \!\left(1 + \frac{L^2 - s^2 - a^2}
{\sqrt{(a^2 + (L - s)^2)(a^2 + (L + s)^2)}} \right).
\]
The last expression converges to $4\pi$ as $L \to \infty$, which establishes the claim.
\end{proof}

The following conclusion rests crucially on quantitative rigidity. 

\begin{lemma} \label{lemma:energydecay}
There exists $\lambda_0>0$ such that, for all $0<|\lambda|<\lambda_0$, there is a sequence $\eps_k \to 0$ satisfying
\[
\lim_{\rho \to 0} \lim_{L \to \infty} \lim_{k \to \infty} 
\int_{D_{\rho} \setminus D_{R_kL}} |\nabla \m_k|^2 \d{x} \le c \lambda^2,
\]
for some constant $c>0$, where $R_k \equiv R_{\eps_k}$.
\end{lemma}

\begin{proof}
Fix $\rho>0$. By the results of Section~\ref{subsec:farfield}, the sequence $(\m_k)$ satisfies
\[
\m_k \rightharpoonup (c,0) \quad \text{weakly in } H^1(\T^2;\S^2)
\text{ and strongly in } H^3(\T^2 \setminus D_{\rho};\S^2),
\]
while, in the sense of measures,
\[
\tfrac{1}{2} |\nabla \m_k|^2 \rightharpoonup \mu \ge 4\pi\delta_0,
\quad \text{and} \quad 
\omega(\m_k) \rightharpoonup -4\pi\delta_0.
\]
Assume that
\[
\int_{\T^2 \setminus D_{\rho}} |\nabla \m_k|^2 \d{x} 
\le \frac{\lambda^2}{|\ln \lambda|^2},
\qquad
\|\m_k - (c,0)\|_{L^{\infty}(\T^2 \setminus D_{\rho})} < \tfrac{1}{2}
\]
for all sufficiently large $k$.  
Introduce the auxiliary map
\[
\bm_k(x) = \frac{\c + \eta(x)(\m_k(x) - \c)}{|\c + \eta(x)(\m_k(x) - \c)|},
\]
where $\c=(c,0)$ and $\eta \in C^{\infty}_c(D_{2 \rho})$ satisfies $\eta \equiv 1$ on $D_{\rho}$, $0 \le \eta \le 1$, and $\|\nabla \eta\|_{L^{\infty}} \lesssim \rho^{-1}$.  
The map $\bm_k \in H^1(\R^2;\S^2)$ has $\deg \bm_k=-1$.  
By \eqref{eq:rigidity}, there exists a harmonic map $\bs{\Psi}_k \in H^1(\R^2;\S^2)$ with $\deg \bs{\Psi}_k=-1$ such that
\begin{equation}\label{ineq:rigidity}
     \int_{\R^2} |\nabla \bm_k - \nabla \bs{\Psi}_k|^2 \d{x} 
     \le c \!\left(\frac12 \int_{\R^2} |\nabla \bm_k|^2 \d{x} - 4 \pi \right) 
     \le c \lambda^2.
\end{equation}
Consequently,
\[
\int_{D_{\rho} \setminus D_{R_k L}} |\nabla \m_k|^2 \d{x} 
\le c \lambda^2 + \int_{D_{\rho} \setminus D_{R_k L}} |\nabla \bs{\Psi}_k|^2 \d{x}.
\]

The map $\bm_k$ satisfies
\[
\int_{D_{2 \rho }\setminus D_{\rho}} |\nabla \bm_k|^2 \d{x} 
\le c \int_{D_{2 \rho }\setminus D_{\rho}} \!\left(\frac{|\m_k - \c|^2}{\rho^2} + |\nabla \m_k|^2\right) \d{x}
\le c \frac{\lambda^2}{|\ln \lambda|^2},
\]
and
\[
\int_{D_{2 \rho }\setminus D_{\rho}} \frac{\bms_{3,k}^2}{\eps_k^2} \d{x}
\le \int_{D_{2 \rho }\setminus D_{\rho}} \frac{m_{3,k}^2}{\eps_k^2} \d{x}
\le \frac{\lambda^2}{|\ln \lambda|^2}.
\]
Therefore,
\[
\int_{\R^2} e_{\lambda,\eps_k} (\bm_k) \d{x}
\le \int_{\T^2} e_{\lambda,\eps_k}(\m_k) \d{x}
+ c\frac{\lambda^2}{|\ln \lambda|^2}.
\]
Combining this bound with the upper energy estimate from Subsection~\ref{subsec:cut-off}, Proposition~5.4 of \cite{DengRaduLamyConformalBimeron} implies that
\[
 \bs{\Psi}_k(z)
 = \Phi\!\left(e^{i \phi_k}
 \frac{(\cos\beta_k+\sin\beta_k)(z+z_k)-a_k(\cos\beta_k-\sin\beta_k)e^{i(\alpha_k+\phi_k)}}
 {(\cos\beta_k-\sin\beta_k)(z+z_k)+a_k(\cos\beta_k+\sin\beta_k)e^{i(\alpha_k+\phi_k)}}\right),
\]
where the parameters satisfy
\begin{equation} \label{eq:rigiditycharacterization}
     \Big|a_k - \tfrac{\lambda \, \eps_k}{ |\ln \lambda|}\Big| 
     \le c\,\tfrac{\lambda \, \eps_k}{|\ln \lambda|^{3/2}}, 
     \qquad 
     |\beta_k| \le \tfrac{\lambda}{\sqrt{|\ln\lambda|}}, 
     \qquad 
     |\alpha_k| \le \tfrac{c}{\sqrt{|\ln\lambda|}}.
\end{equation}
As a consequence,
\[
\int |\nabla \bs{\Psi}_k|^2 
= \int\!\left|\nabla \!\left(\frac{z+z_k-a_k}{z+z_k + a_k}\right) \right|^2 
+ \mathcal{O}\!\left(\tfrac{\lambda}{\sqrt{|\ln \lambda|}}\right).
\]
The error term in the last expression is negligible; therefore, the parameters $\beta_k$, $\alpha_k$, and $\phi_k$ may be set to zero without loss of generality.  
Relation \eqref{eq:rigiditycharacterization}, together with Propositions~\ref{prop:notzero} and~\ref{prop:notinfty}, implies $\tfrac{a_k}{R_k} \to a>0$.

To apply Lemma~\ref{lemma:neckharmonic}, suppose for contradiction that $\tfrac{|z_k|}{R_k} \to \infty$.  
Then Lemma~\ref{lemma:localenergy} yields
\[
\frac{1}{2}\int_{D_{R_kL}(z_k)} |\nabla \bs{\Psi}_k|^2 \d{x} 
= 4 \pi \frac{R_k^2L^2}{R_k^2L^2+a_k^2},
\]
and consequently
\[
\lim_{L \to \infty}\lim_{k \to \infty}
\frac{1}{2}\int_{D_{R_kL}(z_k)} |\nabla \bs{\Psi}_k|^2 \d{x} = 4 \pi.
\]
For $|z_k| < \rho$, the inclusion $D_{R_kL}(z_k) \subset \T^2 \setminus D_{R_kL}(0)$ holds, and therefore
\[
\frac{1}{2}\int_{D_{R_kL}(z_k)} |\nabla \m_k|^2 \d{x} 
\le \frac{1}{2}\int_{\T^2 \setminus D_{R_kL}} |\nabla \m_k|^2 \d{x} 
\le 4 \pi (1+c \lambda^2)-\tfrac{\delta_0^2}{2}
\le 4\pi - \tfrac{\delta_0^2}{4}
\]
for $\lambda_0>0$ sufficiently small.  
It follows that
\[
\lim_{L \to \infty}\lim_{k \to \infty}
\|\nabla \bs{\Psi}_k - \nabla \m_k\|_{L^2(D_{R_kL}(z_k))} 
\ge \tfrac{\delta_0^2}{8 \sqrt{\pi}},
\]
which contradicts \eqref{ineq:rigidity}.  
Hence, the assertion follows.
\end{proof}

The decay of the neck energy established in Lemma~\ref{lemma:energydecay} ensures compactness of the rescaled sequence and precludes the loss of degree through the neck region.  
This property permits passage to the limit and the proof of Theorem~\ref{thm:R2existence}.

\begin{proof}[Proof of Theorem~\ref{thm:R2existence}] 
Up to a subsequence, the blow-up maps 
$\tm_k$ converge weakly in $H^2_{\loc}$ to a limit $\m$.  
Propositions~\ref{prop:notzero} and~\ref{prop:notinfty} imply
that for a subsequence
\[
\frac{\eps_k}{R_k} \to \eps_0>0.
\]
Analogously to the argument in Proposition~\ref{prop:notzero}, the limit map $\m$ is a critical point of $E_{\lambda,\eps_0}$ on $\R^2$.  
Furthermore, the convergence $\omega(\m_k)\rightharpoonup-4\pi\delta_0$ together with Lemma~\ref{lemma:energydecay} implies
\[
|\deg\m + 1|\le c\lambda^2,
\]
and consequently $\deg\m=-1$ for $\lambda_0>0$ sufficiently small.  
\end{proof}

\appendix
\section{Behavior on the annulus}

Let $g$ be defined as in \eqref{eq:trial} and \eqref{ineq:trial}, and set
\begin{align} \label{eq:radiusannulus}
r(|x|) = \frac{g'(|x|)}{1 - \sqrt{1 - g'(|x|)^2}} 
= \frac{1 + \sqrt{1 - g'(|x|)^2}}{g'(|x|)},
\end{align}
for $x \in D_{4} \setminus D_{2}$.  
There exists a constant $C>0$ such that
\begin{align*}
\frac{1}{r(|x|)^{2}} \le |x|^{2}
\quad \text{and} \quad
\frac{r'(|x|)}{r(|x|)^{2}} \le \frac{C}{|x|^{2}}
\end{align*}
for all $x \in D_4 \setminus D_2$.

Indeed, for $s \in (2,4)$, the first inequality follows from \eqref{ineq:trial}:
\begin{align*}
r(s)
=  \frac{1+\sqrt{1-g'(s)^2}}{g'(s)} 
\geq \frac{1+\sqrt{1- \tfrac{4s^2}{(1+s^2)^2}}}{\tfrac{2s}{1+s^2}}
= \frac{1}{s}.
\end{align*}
Regarding the derivative, differentiation of \eqref{eq:radiusannulus} yields
\begin{align*}
r'(s) = - \frac{1+ \sqrt{1-g'(s)^2} }{\sqrt{1-g'(s)^2}\, g'(s)^2}\, g''(s),
\end{align*}
and consequently,
\begin{align*}
\frac{r'(s)}{r(s)^2} 
&= \frac{ \left(1+ \sqrt{1-g'(s)^2}\right)/g'(s)^2}{
\sqrt{1-g'(s)^2}\left(\sqrt{1-g'(s)^2} +1\right)^2/g'(s)^2} \, (-g''(s)) \\
&= -\frac{g''(s)}{\sqrt{1-g'(s)^2}\left(\sqrt{1-g'(s)^2} +1\right)}  
\le \frac{C}{s^2}.
\end{align*}

\begin{lemma}\label{lemma:mRBounds}
Let $\m_{R,a} = \Phi(u_{R,a})$ be constructed as in Subsection~\ref{subsec:cut-off}.  
Then there exists a constant $C>0$ such that
\[
|\nabla \m_{R,a}(x)|^2 \le C \frac{a^2}{|x|^4}
\quad \text{and} \quad 
(\m_{R,a})_3(x)^2 \le C\frac{a^2}{|x|^2}
\]
for all $x \in D_{2R} \setminus D_R$.
\end{lemma}

\begin{proof}
Relation \eqref{eq:radiusannulus} implies, for $r_R(s):=R\, r(\tfrac{s}{R})$ and $R \le s \le 2R$,
\[
\frac{r_R'(s)}{r_R(s)^2} \lesssim \frac{1}{s^2},
\qquad
\frac{1}{r_R(s)^2} \lesssim \frac{s^2}{R^4} \lesssim \frac{1}{s^2}.
\]
Since $|\nabla \m_{R,a}|^2 = 4\frac{|D u_{R,a}|^2}{(1+|u_{R,a}|^2)^2}$ and 
$u_{R,a}(x) = u\!\left(r_R(|x|)\tfrac{x}{|x|}\right)$,
application of the chain rule and \eqref{eq:radiusannulus} yields
\begin{align*}
|\nabla \m_{R,a}|^2 
&\le C \frac{|Du|^2}{(1+|u|^2)^2} 
\bigg|_{r_R(|x|)\tfrac{x}{|x|}}  
\left( r_R'(|x|)^2 + \frac{r_R(|x|)^2}{|x|^2} \right) \\
&\le C \frac{a^2}{(r_R(|x|)^2+a^2)^2}
\left( r_R'(|x|)^2 + \frac{r_R(|x|)^2}{|x|^2} \right) \\
&\le C a^2 
\left( \frac{r_R'(|x|)^2}{r_R(|x|)^4} 
+ \frac{1}{r_R(|x|)^2|x|^2} \right)
\le C \frac{a^2}{|x|^4}. 
\end{align*}
The second estimate follows analogously:
\begin{align*}
(\m_{R,a})_3(x)^2 
&= \frac{(|u_{R,a}(x)|^2-1)^2}{(|u_{R,a}(x)|^2+1)^2}
\le 4 a^2\frac{r_R(|x|)^2}{(r_R(|x|)^2+a^2)^2}
\le 4 \frac{a^2}{r_R(|x|)^2}
\le C\frac{a^2}{|x|^2}.
\end{align*}
The assertion is thus established.
\end{proof}

\section{Energy in complex coordinates}\label{sec:proofdensitylift}
For the lift $\m = \Phi(u)$ it follows from conformality of the stereographic map $\Phi$ that
\begin{align*}
|D\m|^2 =  \frac{4|Du|^2}{(|u|^2+1)^2}  .
\end{align*}
Furthermore, we have $m = \frac{2u}{|u|^2+1}$ and $m_3= \frac{|u|^2-1}{|u|^2+1}$, and, therefore, 
\begin{align*}
(\nabla \cdot m) m_3 
&= 2\left(\nabla \cdot  \frac{u}{|u|^2+1} \right)\frac{|u|^2-1}{|u|^2+1}  \\
&= 2\left[  \frac{(\nabla \cdot u)(|u|^2-1)}{(|u|^2+1)^2} - \frac{(|u|^2-1)(u \cdot \nabla )|u|^2}{(|u|^2+1)^3}    \right] \\
&= 2\left[ -  \frac{(\nabla \cdot u)}{(|u|^2+1)^2}+ {\frac{(\nabla \cdot u)|u|^2-(u \cdot \nabla)|u|^2}{(|u|^2+1)^2}}  + 2\frac{(u \cdot \nabla) |u|^2}{(|u|^2+1)^3}  \right] \\
&= - 2\frac{(\nabla \cdot u)}{(|u|^2+1)^2}  + \nabla \cdot \frac{2|u|^2 u }{(|u|^2+1)^2}  . 
\end{align*}
In the second last step we used
\begin{align*}
\frac{(\nabla \cdot u) |u|^2 - (u \cdot \nabla)|u|^2  }{(|u|^2+1)^2} = u \cdot \nabla \frac{1}{(|u|^2+1)^2} + \nabla \cdot \frac{|u|^2 u }{(|u|^2+1)^2}
\end{align*}
which follows immediately from the product rule and the chain rule 
\begin{align*}
u \cdot \nabla \frac{1}{(|u|^2+1)^2} + \nabla \cdot \frac{|u|^2 u }{(|u|^2+1)^2} &= u \cdot {\nabla \frac{1}{|u|^2+1} }+ \frac{(\nabla \cdot u) |u|^2  }{(|u|^2+1)^2} \\
&=  \frac{(\nabla \cdot u) |u|^2 - (u \cdot \nabla)|u|^2  }{(|u|^2+1)^2} .
\end{align*}

\thanks{Acknowledgment: We thank the referees for their detailed and insightful review.}

\bibliographystyle{abbrv}
\bibliography{bibliography.bib}
\end{document}